\documentclass[leqno]{amsart}

\usepackage{stmaryrd,graphicx}
\usepackage{amssymb,mathrsfs,amsmath,amscd,amsthm,color}
\usepackage{float}
\usepackage[all,cmtip]{xy}
\DeclareMathAlphabet{\mathpzc}{OT1}{pzc}{m}{it}
\usepackage{amsfonts,latexsym,wasysym}
\usepackage[pdftex]{hyperref}

\title{On maps with continuous path lifting}


\author[J. Brazas]{Jeremy Brazas}
\address{West Chester University\\ Department of Mathematics\\ 25 University Avenue\\ West Chester, PA 19383, USA}
\email{jbrazas@wcupa.edu}

\author[A. Mitra]{Atish Mitra}
\address{Montana Technological University\\ Department of Mathematical Sciences\\ 1300 W. Park Street\\
Butte, MT, 59701, USA}
\email{amitra@mtech.edu}
\date{\today}

\newtheorem{theorem}{Theorem}
\newtheorem{lemma}[theorem]{Lemma}
\newtheorem{proposition}[theorem]{Proposition}
\newtheorem{corollary}[theorem]{Corollary}
\theoremstyle{definition}\newtheorem{definition}[theorem]{Definition}

\theoremstyle{definition}\newtheorem{example}[theorem]{Example}
\theoremstyle{definition}
\theoremstyle{definition}\newtheorem{problem}[theorem]{Problem}
\theoremstyle{definition}\newtheorem{remark}[theorem]{Remark}
\newcommand{\ds}{\displaystyle}
\newcommand{\wt}{\widetilde}

\newcommand{\bbr}{\mathbb{R}}

\newcommand{\bbh}{\mathbb{H}}

\newcommand{\bbn}{\mathbb{N}}
\newcommand{\bbq}{\mathbb{Q}}

\newcommand{\ui}{[0,1]}

\newcommand{\tXh}{\widetilde{X}_{H}}

\newcommand{\txh}{\tilde{x}_{H}}

\newcommand{\bfz}{{\bf{0}}}
\newcommand{\pxxo}{P(X,x_0)}
\newcommand{\pionex}{\pi_{1}(X,x_0)}

\newcommand{\scrs}{\mathscr{S}}

\numberwithin{theorem}{section}

\begin{document}
\begin{abstract}
We study a natural generalization of covering projections defined in terms of unique lifting properties. A map $p:E\to X$ has the ``continuous path-covering property" if all paths in $X$ lift uniquely and continuously (rel. basepoint) with respect to the compact-open topology. We show that maps with this property are closely related to fibrations with totally path-disconnected fibers and to the natural quotient topology on the homotopy groups. In particular, the class of maps with the continuous path-covering property lies properly between Hurewicz fibrations and Serre fibrations with totally path-disconnected fibers. We extend the usual classification of covering projections to a classification of maps with the continuous path-covering property in terms of topological $\pi_1$: for any path-connected Hausdorff space $X$, maps $E\to X$ with the continuous path-covering property are classified up to weak equivalence by subgroups $H\leq \pi_1(X,x_0)$ with totally path-disconnected coset space $\pi_1(X,x_0)/H$. Here, ``weak equivalence" refers to an equivalence relation generated by formally inverting bijective weak homotopy equivalences.
\end{abstract}
\maketitle

\section{Introduction}
The breadth of the applications of covering space theory has motivated the development of many useful generalizations. Perhaps most historically prominent is Spanier's classical treatment of Hurewicz fibrations with totally path-disconnected fibers \cite[Chapter 2]{Spanier66}. Unfortunately, since fibrations are defined abstractly in terms of homotopy lifting with respect to arbitrary spaces, it is not possible to extend the usual classification of covering projections to classify these maps up to homeomorphism, i.e. deck transformation. Indeed, there exist bijective fibrations $E\to X$ of metric spaces, which are not homeomorphisms \cite[Example 2.4.8]{Spanier66}. In this paper, we identify and classify a natural class of maps that lies properly between Hurewicz and Serre fibrations with totally path-disconnected fibers, namely, those maps that satisfy the following property: a map $p:E\to X$ has the \textit{continuous path-covering property} if for every $e\in E$, all paths $\alpha:([0,1],0)\to (X,p(e))$ have a unique lift $\wt{\alpha}:(\ui,0)\to (E,e)$ so that the lifting function $\alpha\mapsto \wt{\alpha}$ is continuous with respect to the compact-open topology on based path spaces (Definition \ref{maindefs2}).

We note that there is no single, best, or end-all generalized covering space theory since one must choose the properties and structures that one cares about according to intended applications. The literature on the subject is vast, and we do not attempt to give a complete survey here. We briefly mention that there are many natural approaches closely related to shape theory such as R.H. Fox's theory of overlays \cite{Fox74}. Similar approaches using pro-groups or other algebraic structures in place of the usual fundamental group are considered and compared in \cite{ATHP99,HP98,HPM,Wilkins} and the references therein. Covering theories for categories other than the usual topological category, such as uniform spaces \cite{BP07uniform,BDLM10uniform} and topological groups \cite{BP01topgrp} have also appeared. The approach we take is motivated by the ongoing development of topological methods for studying and applying the algebraic and topological properties of fundamental groups. In particular, our approach is most closely related to the theory of semicoverings \cite{Brsemi,FZ13corefree,KMTSemicover}. Semicoverings and topologized fundamental groups have been used to fill in longstanding gaps in general topological group theory \cite{BrazOpenSubgroupsofFTG} that have not been achieved using purely topological methods. We expect similar applications to follow from the strengthened relationship between covering-theoretic methods and topological group theory developed in the current paper.

Hurewicz fibrations with totally path-disconnected fibers form a class of maps with many nice internal properties that include all inverse limits of covering projections. In recent work, the authors of \cite{CHPGeomofCLS} have identified sufficient compactness conditions on the fibers of such a fibration $p:E\to X$ to ensure that $p$ is equivalent to an inverse limit of finite-sheeted covering maps. When local triviality or compactness conditions on the fibers is not assumed, it becomes an onerous task to verify that a given map $p:E\to X$ with unique lifting of all paths (i.e. with the \textit{path-covering property}) is a Hurewicz fibration. In fact, when $X$ is the closed unit disk and $E$ is locally path-connected, this verification is equivalent to a curiously difficult open problem posed by Jerzy Dydak in \cite{Dydak} (see Problem \ref{dydaksproblem}). The apparent difficulty stems from the fact that if one does not already know that $p$ is an inverse limit of fibrations, then one must verify the homotopy lifting property with respect to \textit{all} topological spaces $Z$. Even if one restricts to a convenient category of spaces, one must still address spaces $Z$ in which convergent nets are not always realized as the endpoints of some convergent net of paths. Since our goal is to obtain a theory that supports application of (topologized) homotopy groups, it is clear that we must weaken the highly demanding definition of a ``Hurewicz fibration."

Serre fibrations with totally path-disconnected fibers form a significantly larger class of maps than their Hurewicz counterparts. Included among these are the generalized regular covering maps defined originally by Fischer and Zastrow in \cite{FZ07} and later in \cite{BDLM08} using a different but ultimately equivalent approach. Such maps were extended to the non-regular case in \cite{Brazcat} and characterized completely within the $\pi_1$-subgroup lattice for metric spaces in \cite{BFTestMap}. Such generalized covering maps provide a theory, which has been proven to retain the largest possible $\pi_1$-subgroup lattice among any other theory that employs homotopy lifting \cite{Brazcat}. For instance, such maps can retain non-trivial information about the $\pi_1$-subgroup lattice even for spaces with trivial shape-type. Consequently, the intended application of this locally path-connected approach is to provide a highly refined theory to aid the progressive work on the infinitary-algebraic structure of homotopy groups of Peano continua and other locally path-connected spaces.

In Section \ref{sectionbasictheory}, we develop the basic theory of the continuous path-covering property. We observe that, just as with fibrations, maps with the continuous path-covering property are closed under composition, infinite products, pullbacks, inverse limits. Moreover, these maps even enjoy the ``2-of-3 property" (see Lemma \ref{compositionlemma}). In Section \ref{sectioncomparision}, we give a general comparison of maps with unique lifting properties. 

\begin{theorem}\label{fibrationtheorem}
Consider the following properties of a map $p:E\to X$.
\begin{enumerate}
\item $p$ is a Hurewicz fibration with totally path-disconnected fibers,
\item $p$ has the continuous path-covering property,
\item $p$ is a Serre fibration with totally path-disconnected fibers,
\item $p$ has the path-covering property.
\end{enumerate}
Then (1) $\Rightarrow$ (2) $\Rightarrow$ (3) $\Rightarrow$ (4).
\end{theorem}

Following the proof of Theorem \ref{fibrationtheorem}, we provide examples to show that the conditionals (1) $\Rightarrow$ (2) and (2) $\Rightarrow$ (3) are not reversible. Hence, the continuous path-covering property lies properly between Hurewicz and Serre fibrations with totally path-disconnected fibers. To promote the fundamental nature of Dydak's Problem, referenced above, we show that an affirmative answer to this problem is equivalent to the converse of (3) $\Rightarrow$ (4). We also identify a consequence of a positive answer to Dydak's Problem within the context of classical covering space theory (Corollary \ref{corollarytodydak}).

In Section \ref{sectioninducedhomo}, we consider the homotopy groups $\pi_n(X,x_0)$ with the natural quotient topology inherited from the compact-open topology on the $n$-th iterated loop space $\Omega^n(X,x_0)$ \cite{GHMMTopHomGrps,Fab11}. With this topology, the homotopy groups become functors to the category of quasitopological groups and continuous homomorphisms. We show that a map $p:E\to X$ with the continuous path-covering property induces a closed embedding on fundamental groups and an isomorphism of quasitopological abelian groups on the higher homotopy groups (Theorem \ref{closedembeddingonpi1}). Moreover, for any $e\in E$, the coset space $\pionex/p_{\#}(\pi_1(E,e))$ is totally path-disconnected and is homeomorphic to the fibers of $p$ if and only if the path endpoint-evaluation map $ev_1:P(E,e_0)\to E$, $ev_1(\alpha)=\alpha(1)$ is a topological quotient map (Theorem \ref{liftingcorrespondencetheorem}). The topological property ``$ev_1:P(Z,z)\to Z$ is quotient" of a space $Z$ is a generalization of a space being ``path connected and local path-connected," which arises naturally within in our work, sometimes within biconditional statements. 

Section \ref{sectionclassification} is dedicated to a proof of our classification theorem. Just as with Serre/Hurewicz fibrations with totally path-disconnected fibers, it is not possible to extend the traditional classification of covering projections in a way that classifies all continuous path-covering property up to \textit{homeomorphism} (see Example \ref{zeemansexample}). This reality is due to the fact that these maps are defined categorically in terms of lifting properties rather than in terms of local triviality conditions. Therefore, to identify a suitable and practical classification, we employ a technique from model category theory \cite{Hovey}, namely, ``localization at the weak homotopy equivalences," which refers to the formal inversion of weak homotopy equivalences to generate an equivalence relation. A \textit{simple weak equivalence} between two maps $p_i:E_i\to X$ $i\in\{1,2\}$ with the continuous path-covering property consists of a commutative diagram
\[\xymatrix{
E_1 \ar[dr]_-{p_1} & E_3 \ar[l]_-{f_1}  \ar[d]^-{p_3} \ar[r]^-{f_2} & E_2  \ar[dl]^-{p_2}\\ & X
}\]
where $p_3$ also has the continuous path-covering property and $f_1,f_2$ are \textit{bijective} weak homotopy equivalences (in fact, they necessarily induce topological isomorphisms on all homotopy groups). Then $p_1,p_2$ are \textit{weakly equivalent} if they are connected by a finite sequence of simple weak equivalences. Our main classification result is the following.

\begin{theorem}\label{mainresult1}
Suppose $(X,x_0)$ is a path-connected Hausdorff space and $H\leq \pionex$. There exists a map $p:(E,e_0)\to (X,x_0)$ with the continuous path-covering property, unique up to weak equivalence, such that $p_{\#}(\pi_1(E,e_0))=H$ if and only if $\pionex/H$ is totally path disconnected. Moreover,
\begin{enumerate}
\item among the maps $p:E\to X$ for which $ev_1:P(E,e)\to E$ is quotient, the maps $p$ are classified up to equivalence,
\item every map $p:E\to X$ with the continuous path-covering property is weakly equivalent to another $p':E'\to X$ where $ev_1(E',e')\to E'$ is quotient.
\end{enumerate}
\end{theorem}

The proof of the existence portion of Theorem \ref{mainresult1} requires a careful analysis of the natural quotient construction of covering spaces. A surprising consequence of our proof of (2) of Theorem \ref{mainresult1} is that two weakly equivalent maps can always be related by a single simple weak equivalence. By employing topologized fundamental group theory \cite{BFqtpfg}, another immediate consequence of this theorem is that any space $X$ whose fundamental group naturally injects into the first shape group admits a ``universal weak equivalence class," i.e. a map $E\to X$ with the continuous path-covering property where $E$ is simply connected. For instance, all one-dimensional metric spaces and planar sets admit such a map.

If $p:E\to X$ has the continuous path-covering property and non-discrete fibers, then $E$ will rarely be locally path-connected. In a sense, this is the price one must pay to ensure that paths lift continuously. However, (2) of Theorem \ref{mainresult1} still allows for a certainly level of topological control, namely, that we may always choose $E$ to have the slightly weaker property that $ev_1:P(E,e_0)\to E$ is quotient. In summary, Parts (1) and (2) of Theorem \ref{mainresult1} show that our classification up to weak equivalence restricts to a more traditional classification ``up to homeomorphism" when we restrict to the category of spaces $E$ for which $ev_1:P(E,e)\to E$ is quotient.

Finally, in Section \ref{sectionstructure}, we prove the following theorem, which identifies natural situations where a weak equivalence class may be represented by an inverse limit of covering projections/semicovering maps.

\begin{theorem}\label{mainresult2}
Suppose $p:(E,e_0)\to (X,x_0)$ has the continuous path-covering property where $X$ is locally path connected and $H=p_{\#}(\pi_1(E,e_0))$ is a normal subgroup of $\pionex$.
\begin{enumerate}
\item If $\pionex/H$ is a compact group, then $p$ is weakly equivalent to an inverse limit of finite-sheeted, regular covering projections.
\item If $E$ is simply connected, i.e. $H=1$, and $\pionex$ is locally compact, then $p$ is weakly equivalent to an inverse limit of semicovering maps.
\end{enumerate}
\end{theorem}
We conclude the paper, in Section \ref{sectiondiagram}, with a single diagramatic summary of our results. This diagram, which incorporates our main results, illustrates the relationships between the maps considered throughout this paper as well as the relationships between these maps and the topological-algebraic properties of $\pionex$.
\section{Notation and Preliminaries}\label{sectionprelims}

\subsection{Mapping Spaces and Path Spaces}

If $X$ and $Y$ are spaces, $Y^X$ denotes the set of continuous maps $X\to Y$ with the compact-open topology and for $A\subseteq X$, $B\subseteq Y$, $(Y,B)^{(X,A)}$ denotes the subspace of relative maps $f\in Y^X$ satisfying $f(A)\subseteq B$. In particular, $\Omega^{n}(X,x)$, $n\in\bbn$ will denote the relative mapping space $(X,x)^{(I^n,\partial I^n)}$. The constant map $X\to Y$ at a point $y\in Y$ will be denoted $c_y$. The free path space $X^{I}$ where $I=\ui$ is the closed unit interval will be denoted by $P(X)$. The subspace $(X,x)^{(I,0)}$ of paths $\alpha:I\to X$ starting at $x$ as $P(X,x)$; note that $\Omega(X,x)\subseteq P(X,x)$. If $\alpha,\beta:I\to X$ are paths such that $\alpha(1)=\beta(0)$, then $\alpha\cdot\beta:I\to X$ denotes the standard concatenation and $\alpha^{-}(t)=\alpha(1-t)$ denotes the reverse path of $\alpha$. 

If $f:X\to Y$ is a map, then $P(f):P(X)\to P(Y)$ denotes the induced map $P(f)(\alpha)=f\circ\alpha$, which also restricts to a map $P(X,x)\to P(X,f(x))$. The \textit{endpoint evaluation map} is the map $ev_1:P(X,x)\to X$, $ev_1(\alpha)=\alpha(1)$, which is continuous and onto when $X$ is path connected. If $X$ is locally path-connected, then $ev_1$ is an open surjection. Consider the property of $X$: ``$ev_1:P(X,x)\to X$ is a topological quotient map," which holds for all points $x\in X$ if it holds for at least one. This property is a natural generalization of the joint property ``path connected and locally path connected" and will appear throughout this paper. Path-connected and non-locally path-connected spaces for which $ev_1$ is quotient include all non-locally path-connected, contractible spaces and many spaces used in the theory and applications of generalized covering space theories and topologized fundamental groups, c.f. \cite{Brsemi,BrazOpenSubgroupsofFTG}.

We refer to \cite{Spanier66} as a standard reference on covering space theory. We consider the following properties, each of which is held by all covering projections.

\begin{definition}\label{maindefs1}
Let $E$ and $X$ be topological spaces.
\begin{enumerate}
\item A map $p:E\to X$ has the \textit{unique path-lifting property} if for each $e\in E$, the induced map $P(p):P(E,e)\to P(X,p(e))$ is injective.
\item A map $p:E\to X$ has the \textit{path-covering property} if for each $e\in E$, the induced map $P(p):P(E,e)\to P(X,p(e))$ is bijective.
\end{enumerate}
\end{definition}

We will always assume the spaces $E$ and $X$ are non-empty and path connected. If $p:E\to X$ has the path-covering property, $E\neq \emptyset$, and $X$ is path-connected, then $p$ must necessarily be surjective. 

\subsection{Topologized homotopy groups}

The $n$-th homotopy group $\pi_n(X,x_0)$ will be equipped with the natural quotient topology inherited from $\Omega^n(X,x_0)$ so that the natural map $\pi:\Omega^n(X,x_0)\to\pi_n(X,x_0)$, $\pi(\alpha)=[\alpha]$ sending a map to its homotopy class is a topological quotient map. It is known that $\pi_n(X,x_0)$ is a quasitopological group in the sense that inversion is continuous and left and right translations $[\alpha]\mapsto [\alpha][\beta]$ and $[\alpha]\mapsto [\beta][\alpha]$ for fixed $\beta\in\Omega^n(X,x_0)$ are continuous. Although $\pi_n(X,x)$ can fail to be a topological group for any $n\geq 1$ \cite{Fab10,Fab11}, it is a homogeneous space, which is discrete if $X$ is locally contractible \cite{CM,GHMMTopHomGrps}.

If $f:(X,x)\to (Y,y)$ is a based map, then the homomorphism $f_{\#}:\pi_n(X,x)\to \pi_n(Y,y)$ is continuous. An isomorphism in the category of quasitopological groups is a group isomorphism, which is also a homeomorphism. If $f$ induces an isomorphism $f_{\#}:\pi_n(X,x)\to \pi_n(Y,y)$ of quasitopological groups for all $n\geq 1$, then we call $f$ a \textit{weak topological homotopy equivalence}.

If $H\leq \pionex$ is a subgroup, the coset space $\pionex/H$ inherits the quotient topology from $\pionex$. The translation homeomorphism $[\alpha]\mapsto [\alpha\cdot\beta]$ of $\pionex$ descends to a homeomorphism $H[\alpha]\mapsto H[\alpha\cdot\beta]$ on $\pionex/H$. Hence, $\pionex/H$ is a homogeneous space, which is $T_1$ (resp. discrete) if and only if $H$ is closed (resp. open). We refer to \cite{BFqtpfg} for more on $\pi_1$ with the quotient topology.

\subsection{Fibrations with the Unique Path-Lifting Property}

\begin{definition}
A map $p:E\to X$ has the \textit{homotopy lifting property} with respect to a space $Z$ if for every pair of maps $f:Z\to E$, $g:Z\times I\to X$ such that $p\circ f (z)=g(z,0)$, there is a map $\wt{g}:Z\times I\to E$ such that $p\circ \wt{g}=g$. A \textit{Hurewicz fibration} is a map with the homotopy lifting property with respect to all topological spaces. A \textit{Serre fibration} is a map with the homotopy lifting property with respect to $I^n$ for all $n\geq 0$.
\end{definition}

Every covering projection in the classical sense is a Hurewicz fibration with discrete fibers and every Hurewicz fibration is a Serre fibration.

\begin{lemma}\cite[Proof of 2.2.5]{Spanier66}\label{serrecharlemma}
A Serre fibration $p:E\to X$ has the unique path-lifting property if and only if every fiber of $p$ is totally path disconnected.
\end{lemma}

If a map $p:E\to X$ has the path-covering property and also the homotopy lifting property with respect to $I$, then all path-homotopies in $X$ lift uniquely (rel. basepoint) to path-homotopies in $E$. Hence, standard arguments in covering space theory give the following lemma.

\begin{lemma}\label{basicliftingpropertieslemma}
Suppose $p:E\to X$ has the path-covering property and the homotopy lifting property with respect to $I$, e.g. if $p$ is a Serre fibration with totally path-disconnected fibers. Then for all $e\in E$
\begin{enumerate}
\item the induced homomorphism $p_{\#}:\pi_1(E,e)\to \pi_1(X,p(e))$ is injective,
\item the unique lift $\wt{\alpha}\in P(E,e)$ of a loop $\alpha\in\Omega(X,p(e))$ is a loop based at $e$ if and only if $[\alpha]\in p_{\#}(\pi_1(E,e))$.
\end{enumerate}
Moreover, if $p_{\#}:\pi_1(E,e)\to \pi_1(X,p(e))$ is surjective, then $p$ is a continuous bijection.
\end{lemma}

\section{Maps with the continuous lifting property}\label{sectionbasictheory}

The goal of this section is to develop the basic properties of maps with the following property.

\begin{definition}\label{maindefs2}
A map $p:E\to X$ has the \textit{continuous path-covering property} if for every $e\in E$, the induced function $P(p):P(E,e)\to P(X,p(e))$ is a homeomorphism.
\end{definition}

\begin{remark}
Certainly, we have: 
\[\text{(1) continuous path-covering $\Rightarrow$ (2) path-covering $\Rightarrow $ (3) unique path-lifting.}\]
However, none of the reverse implications hold in general. For instance, any restriction of a covering projection, which is not a covering projection itself satisfies (3) but not (2). The generalized universal covering of the Hawaiian earring constructed in \cite{FZ07} satisfies (2) but not (1) and is described in more detail below (see Example \ref{heexample}).
\end{remark}

\begin{proposition}\label{tpdfibersprop}
If $p:E\to X$ has the unique path-lifting property, then for every $x\in X$ the fiber $p^{-1}(x)$ is totally path disconnected (and hence $T_1$). Moreover, if $X$ is $T_1$, then so is $E$.
\end{proposition}

\begin{proof}
If $p^{-1}(x)$ admitted a non-constant path $\alpha:I\to p^{-1}(x)$, then $\alpha$ and the constant path at $\alpha(0)$ are distinct lifts of $p\circ\alpha$; a violation of unique path lifting. Every totally path-disconnected space is $T_1$ since a non-$T_1$ space must contain a homeomorphic copy of a non-discrete 2-point space, which is necessarily path-connected. Moreover, if $p(e)=x$ and $\{x\}$ is closed, then $\{e\}$ is closed in the closed fiber $p^{-1}(x)$ and thus closed in $E$. 
\end{proof}

Recall that an infinite product of covering projections need not be a covering projection, e.g. any infinite power of the exponential map $\bbr\to S^1$ provides an example.

\begin{lemma}
The following classes of maps are closed under arbitrary direct products:
\begin{enumerate}
\item maps with the unique path-lifting property,
\item maps with the path-covering property,
\item maps with the continuous path-covering property,
\item Hurewicz fibrations with totally path-disconnected fibers,
\item Serre fibrations with totally path-disconnected fibers.
\end{enumerate}
\end{lemma}

\begin{proof}
The first three cases are clear since the based path-space functors $(X,x)\mapsto P(X,x)$ preserve direct products \cite[Proposition 3.4.5]{engelking}. The last two follow from the fact that (1) if maps $p_j:X_j\to Y_j$ have the homotopy lifting property with respect to a space $Z$, then so does the product map $\prod_{j}p_j$ and (2) products of totally path-disconnected spaces are totally path disconnected.
\end{proof}

Covering projections also fail to be closed under function composition. Maps with the continuous path-covering property, in fact, satisfy the ``two-out-of-three" condition in the next lemma. Since the various parts of the statement are straightforward to verify from the definitions, we omit the proof.

\begin{lemma}\label{compositionlemma}
Suppose $f:X\to Y$ and $g:Y\to Z$ are maps of non-empty, path-connected spaces.
\begin{enumerate}
\item If $f$ and $g$ have the continuous path-covering property, then so does $g\circ f$.
\item If $g$ and $g\circ f$ have the continuous path-covering property, then so does $f$.
\item If $g$ is surjective and $f$ and $g\circ f$ have the continuous path-covering property, then so does $g$.
\end{enumerate}
Moreover, the statement holds if we replace ``continuous path-covering property" with ``path-covering property."
\end{lemma}

\begin{remark}
Serre/Hurewicz fibrations are closed under composition and if $g$ and $g\circ f$ are Serre/Hurewicz fibrations, then so is $f$. Although the authors do not know of a counterexample, we find it unlikely that these classes of maps are closed under the third combination.
\end{remark}

The \textit{cone} over a space $X$ is the quotient space $CX=X\times I/X\times \{0\}$. The point $v_0\in CX$, which is the image of $X\times \{0\}$, is taken to be the basepoint of $CX$.

\begin{definition}
Let $(J,\leq)$ be a directed set and $K=J\cup\{\infty\}$ be the space obtained by adding one maximal point. Give $K$ the topology generated by the sets $\{k\}$ and $V_{k}=\{\infty\}\cup\{j\in J\mid j>k\}$ for $k<\infty$. The \textit{directed arc-fan over }$J$ is the cone over $K$, i.e. the quotient space $F(J)=K\times I/K\times \{0\}$ with basepoint $v_0$. We typically will identify $K\times (0,1]$ with it's image in $F(J)$.
\end{definition}

\begin{remark}\label{netstofansremark}
Standard exponential laws for spaces imply that the convergent nets $\{\alpha_j\}_{j\in J}\to \alpha$ in $P(X,x)$ are in bijective correspondence with based maps $(F(J),v_0)\to (X,x)$. Hence, a map $p:E\to X$ has the continuous path-covering property if and only if for every $e\in E$ and directed set $J$, the map $F:(E,e)^{(F(J),v_0)}\to (X,p(e))^{(F(J),v_0)}$, $F(\beta)=p\circ\beta$ is a bijection. If $X$ is a metric space, then the compact-open topology on $P(X,x)$ agrees with the topology of uniform convergence and one need only consider maps $F(\omega)\to X$ on the directed fan $F(\omega)$ indexed by the natural numbers.
\end{remark}

In the remainder of this section, we show that maps $E\to X$ with the continuous path-covering property also lift maps $Z\to X$ from many other spaces $Z$ both uniquely and continuously.

\begin{lemma}\label{coneliftinglemma}
Let $Z$ be a compact Hausdorff space and $z\in Z$. If $p:E\to X$ has the continuous path-covering property, then for every $e\in E$, the induced map $F:(E,e)^{(CZ,v_0)}\to (X,p(e))^{(CZ,v_0)}$, $F(\beta)=p\circ\beta$ is a homeomorphism. 
\end{lemma}

\begin{proof}
Fix $e\in E$. By assumption, $P(p):P(E,e)\to P(X,p(e))$ is a homeomorphism so it follows from functorality that $P(p)^Z:P(E,e)^Z\to P(X,p(e))^Z$ is a homeomorphism. We call upon some elementary facts related to exponential laws in the category of topological spaces. Since $Z$ is compact Hausdorff, for any based space $(A,a)$, the mapping space $P(A,a)^Z$ is naturally homeomorphic to the relative mapping space $(A,a)^{(Z\times I,Z\times \{0\})}$, which is, in turn, naturally homeomorphic to the based mapping space $(A,a)^{(CZ,v_0)}$. It follows that $F:(E,e)^{(CZ,v_0)}\to (X,p(e))^{(CZ,v_0)}$ is a homeomorphism.
\[\xymatrix{
P(E,e)^{Z} \ar[d]_{\cong} \ar[rr]^-{P(p)^Z} & & P(X,p(e))^{Z} \ar[d]^{\cong}\\
(E,e)^{(Z\times I,Z\times \{0\})} \ar[d]_{\cong} \ar[rr]^-{p^{(Z\times I,Z\times \{0\})}} && (X,p(e))^{(Z\times I,Z\times \{0\})} \ar[d]^{\cong}\\
(E,e)^{(CZ,v_0)} \ar[rr]_-{F} && (X,p(e))^{(CZ,v_0)}
}\]
\end{proof}

Since $I^{n+1}\cong CI^n$, we obtain the following corollary where $\bfz\in I^n$ denotes the origin.

\begin{corollary}\label{continuoushomotopycorollary}
If $p:E\to X$ has the continuous path-covering property, then for every $e\in E$ and $n\in\bbn$, the induced map $p^{(I^n,\bfz)}:(E,e)^{(I^n,\bfz)}\to (X,p(e))^{(I^n,\bfz)}$ is a homeomorphism.
\end{corollary}

\begin{remark}\label{liftremark}
Note that Corollary \ref{continuoushomotopycorollary} implies that maps with the continuous path-covering property have the homotopy lifting property with respect to $I$ and thus the conclusions of Lemma \ref{basicliftingpropertieslemma} apply to all such maps. 
\end{remark}

The following lemma is essentially \cite[Lemma 2.5]{Brsemi}. We give a direct statement and proof that avoids groupoid terminology.

\begin{lemma}\label{liftinglemma}
If $p:E\to X$ has the continuous path-covering property, $p(e_0)=x_0$, and $(Z,z_0)$ is a based space such that $ev_1:P(Z,z_0)\to Z$ is quotient, then a map $f:(Z,z_0)\to (X,x_0)$ has a unique continuous lift $\wt{f}:(Z,z_0)\to (E,e_0)$ if and only if $f_{\#}(\pi_1(Z,z_0))\leq p_{\#}(\pi_1(E,e_0))$.
\end{lemma}

\begin{proof}
By Corollary \ref{continuoushomotopycorollary}, $p$ uniquely lifts paths and path-homotopies. Hence, the condition $f_{\#}(\pi_1(Z,z_0))\leq p_{\#}(\pi_1(E,e_0))$ is equivalent to the well-definedness of the lift function $\wt{f}$ with the following standard definition: for $z\in Z$, let $\gamma\in P(Z,z_0)$ be a path ending at $z$, $\wt{f\circ\gamma}\in P(E,e_0)$ be the unique lift of $f\circ\gamma$, and set $\wt{f}(z)=\wt{f\circ\gamma}(1)$. Since $ev_1:P(Z,z_0)\to Z$ is quotient, $Z$ is path connected. This makes the uniqueness of $\wt{f}$ clear once we verify continuity. Let $P(p)^{-1}:P(X,x_0)\to P(E,e_0)$ be the continuous lifting homeomorphism and consider the following diagram for which the commutativity is equivalent to the definition of $\wt{f}$.
\[\xymatrix{
P(Z,z_0) \ar[d]^{ev_1} \ar[r]^-{P(f)} & P(X,x_0) \ar[r]^-{P(p)^{-1}}_-{\cong} & P(E,e_0) \ar[d]_{ev_1} \\
Z \ar[rr]_{\wt{f}} && E
}\]
Since the top composition $P(Z,z_0)\to E$ is continuous and $ev_1:P(Z,z_0)\to Z$ is assumed to be quotient, $\wt{f}$ is continuous by the universal property of quotient maps.
\end{proof}

\begin{corollary}\label{cpcequivalencetheorem}
If $p_1:(E_1,e_1)\to (X,x_0)$ and $p_2:(E_2,e_2)\to (X,x_0)$ are maps with the continuous path-covering property such that
\begin{enumerate}
\item $ev_1:P(E_1,e_1)\to E_1$ and $ev_1:P(E_2,e_2)\to E_2$ are quotient,
\item $(p_1)_{\#}(\pi_1(E_1,e_1))=(p_2)_{\#}(\pi_1(E_2,e_2))$,
\end{enumerate}
then there exists a unique homeomorphism $h:(E_1,e_1)\to (E_2,e_2)$ such that $p_2\circ h=p_1$.
\end{corollary}

\begin{theorem}\label{cogenliftingtheorem}
Suppose $p:E\to X$ has the continuous path-covering property, $e\in E$, and $(Z,z)$ is a path-connected space. Consider the map $F:(E,e)^{(Z,z)}\to (X,p(e))^{(Z,z)}$ given by $F(f)=p\circ f$.
\begin{enumerate}
\item If $Z$ is contractible, then $F$ is bijective,
\item If $Z$ is contractible and compact Hausdorff, then $F$ is a homeomorphism.
\end{enumerate}
\end{theorem}

\begin{proof}
Since $Z$ is contractible, there is a section $s:Z\to P(Z,z)$ to the evaluation map $ev_1:P(Z,z)\to Z$. Thus, the latter is a quotient map. The injectivity of $F$ follows from the fact that $Z$ is path connected and $p$ has the unique path-lifting property. Since $Z$ is simply connected and $ev_1:P(Z,z)\to Z$ is quotient, Lemma \ref{liftinglemma} applies to give the surjectivity of $F$.

Note that Lemma \ref{coneliftinglemma} proves (2) in the case where $Z$ is a cone. For general contractible $Z$, there is a retraction $r:CZ\to Z$ such that $r(v_0)=z$. This means that for every space $(A,a)$, the induced map $R:(A,a)^{(Z,z)}\to (A,a)^{(CZ,v_0)}$, $R(g)=g\circ r$ is a section and therefore a topological embedding. Consider the naturality diagram
\[\xymatrix{
(E,e)^{(Z,z)} \ar[r]^-{F} \ar[d]_-{R} & (X,p(e))^{(Z,z)} \ar[d]^-{R}\\
(E,e)^{(CZ,v_0)} \ar[r]_-{p^{(CZ,v_0)}} & (X,p(e))^{(CZ,v_0)}
}\]
where the vertical maps are embeddings and the bottom map is a homeomorphism (recall Lemma \ref{coneliftinglemma}). It follows that $F$ is a topological embedding. To see that $F$ is onto, consider a map $g:(Z,z)\to (X,p(e))$. Since the bottom horizontal map is bijective, we may find a unique map $k:(CZ,v_0)\to (E,e)$ such that $p\circ k=R(g)=g\circ r$. If $i:Z\to CZ$ is the canonical inclusion, set $\wt{g}=k\circ i$. First, we must check that $\wt{g}(z)=e$ to consider $\wt{g}$ a based map. Consider the path $\alpha:I\to CZ$, $\alpha(t)=(z,t)$ from $v_0$ to the basepoint $i(z)$ and note that $r\circ\alpha:I\to Z$ is a null-homotopic loop based at $z$. Observe that $k\circ\alpha:I\to E$ starts at $e$ and $p\circ k\circ \alpha=g\circ r\circ\alpha $ is a null-homotopic loop based at $p(e)$. By Lemma \ref{basicliftingpropertieslemma}, the unique lift $k\circ\alpha$ of $g\circ r\circ\alpha$ starting at $e$ must be a loop, i.e. $k(v_0)=k(z)$. Thus $\wt{g}(z)=\wt{g}(r(v_0))=k\circ i(r(v_0))=k(z)=k(v_0)=e$. Finally, we have $F(\wt{g})=p\circ k\circ i=g\circ r\circ i=g$, completing the proof that $F:(E,e)^{(Z,z)}\to (X,p(e))^{(Z,z)}$ is onto.
\end{proof}

\begin{theorem}\label{freepathspacesthm}
If $p:E\to X$ has the continuous path-covering property, then so does the map $P(p):P(E)\to P(X)$ induced on free path spaces.
\end{theorem}

\begin{proof}
Let $\alpha\in P(E)$ and $e=\alpha(0)$. By Corollary \ref{continuoushomotopycorollary}, the induced map $F_2:(E,e)^{(I^2,\bfz)}\to (X,p(e))^{(I^2,\bfz)}$ is a homeomorphism. Let $A=\{h\in (E,e)^{(I^2,\bfz)}\mid h(t,0)=\alpha(t)\}$ and similarly $B=\{h\in (X,p(e))^{(I^2,\bfz)}\mid h(t,0)=p\circ\alpha(t)\}$. Note that $F_2$ maps $A$ into $B$. Since $F_2$ is surjective, if $h\in B$, then there is a lift $\wt{h}\in (E,e)^{(I^2,\bfz)}$ such that $\wt{h}(t,0)$ is a path satisfying $\wt{h}(0,0)=e$ and $p\circ \wt{h}(t,0)=p\circ \alpha(t)$. Since $p$ has the unique path-lifting property, we have $\wt{h}(t,0)=\alpha(t)$ and thus $\wt{h}\in A$. It follows that $F_2$ maps $A$ homeomorphically onto $B$. Restricting the exponential law naturality diagram on the left gives the commutativity of the diagram on the right. 
\[\xymatrix{
P(P(E)) \ar[r]^-{P(P(p))} \ar[d]_-{\cong} & P(P(X)) \ar[d]^-{\cong} & & P(P(E),\alpha) \ar[r]^-{P(P(p))} \ar[d]_-{\cong} & P(P(X),p\circ \alpha) \ar[d]^-{\cong}\\
E^{I^2}
\ar[r]_-{p^{I^2}} & X^{I^2} & &  A \ar[r]_-{(F_2)|_{A}}^-{\cong} & B
}\]
It follows that the map $P(P(E),\alpha)\to P(P(X),p\circ\alpha)$ on based path spaces induced by $p$ is a homeomorphism.
\end{proof}

\section{Comparison to fibrations and Dydak's Problem}\label{sectioncomparision}

\begin{lemma}\label{hurfibrationlemma}
A Hurewicz fibration has totally path-disconnected fibers if and only if it has the continuous path-covering property.
\end{lemma}

\begin{proof}
The ``if" direction is a classical result (recall Lemma \ref{serrecharlemma}). For the ``only if" direction, suppose $p:E\to X$ is a Hurewicz fibration with totally path-disconnected fibers, i.e. with the unique path-lifting property. Since $p$ has the homotopy lifting property with respect to a point, $p$ has the path-covering property. Therefore, it suffices to verify the continuity of path lifting. According to Remark \ref{netstofansremark}, $p$ will have the continuous path-lifting property if we can show that maps from directed arc-fans can always be lifted (rel. basepoint). Suppose $J$ is an infinite directed set and $K=J\cup \{\infty\}$ as before. Let $e\in E$ and $f:(F(J),v_0)\to (X,f(e))$ be a map on the directed arc-fan. Let $q:K\times I\to F(J)$ be the quotient map, $g=f\circ q$, and $\wt{g}_0:K\times\{0\}\to E$ be the constant map at $e$.
\[\xymatrix{
K\times\{0\} \ar[rr]^-{\wt{g}_0} \ar[d] && E \ar[d]^-{p}\\
K\times I \ar@{-->}[urr]^-{\wt{g}} \ar@/_1.5pc/[rr]_-{g} \ar[r]_-{q} & F(J)  \ar[r]_-{f} \ar@{-->}[ur]_-{\wt{f}} & X
}\]
Since a Hurewicz fibration $p:E\to X$ has the homotopy lifting property with respect to $K$, there is a map $\wt{g}:K\times I\to X$ making the above diagram commute. The commutativity of the upper left triangle in the diagram implies that $\wt{g}$ induces a map $\wt{f}:(F(J),v_0)\to (E,e)$ on the quotient such that $p\circ \wt{f}=f$.
\end{proof}

Recall the statement of Theorem \ref{fibrationtheorem} from the introduction.

\begin{proof}[Proof of Theorem \ref{fibrationtheorem}]
(1) $\Rightarrow$ (2) follows from Lemma \ref{hurfibrationlemma}.

For (2) $\Rightarrow$ (3), suppose $p:E\to X$ has the continuous path-covering property. By Proposition \ref{tpdfibersprop}, it suffices to show that $p$ has the homotopy lifting property with respect to $I^n$ for $n\geq 1$. Note that $p:E\to X$ has the homotopy lifting property with respect to a locally compact Hausdorff space $Z$ if and only if for every map $f:Z\to E$, the induced map $P(p^Z):P(E^Z,f)\to P(X^Z,p\circ f)$ is surjective. By inductively applying Theorem \ref{freepathspacesthm} with the exponential homeomorphism $(W^{I^n})^{I}\cong W^{I^{n+1}}$, we have that for every $n\geq 1$ and map $f:I^n\to E$, the map $P(p^{I^n}):P(E^{I^n},f)\to P(X^{I^n},p\circ f)$ is a homeomorphism. Therefore, $p$ has the homotopy lifting property with respect to all cubes $I^n$, $n\geq 0$ and is a Serre fibration.

(3) $\Rightarrow$ (4) follows from Lemma \ref{serrecharlemma}.
\end{proof}

We briefly mention a relevant consequence that also appears in \cite{Brsemi}. Since every covering projection is a Hurewicz fibration with discrete fibers \cite[Theorem 2.2.3]{Spanier66}, we have the following.

\begin{corollary}\label{coveringmap}\cite{Brsemi} Every covering projection has the continuous path-covering property.
\end{corollary}

The previous corollary allows us to answer the question: when do we know that a map with the continuous path-covering property is, in fact, a genuine covering projection? Our answer provides a generalization of the classical result \cite[Theorem 2.5.10]{Spanier66}.
\begin{corollary}\label{coveringcorollary1}
Suppose $X$ is locally path-connected and semilocally simply connected and $p:E\to X$ is a map with the continuous path-covering property. If $ev_1:P(E,e_0)\to E$ is quotient (for instance, if $E$ is locally path-connected), then $p$ is a covering projection.
\end{corollary}

\begin{proof}
The hypotheses on $X$ imply that there exists a covering projection $q:E'\to X$ and $e_0'\in E'$ such that $q_{\#}(\pi_1(E',e_0'))=p_{\#}(\pi_1(E,e_0))$. By Corollary \ref{coveringmap}, $q$ has the continuous path-covering property. Since $ev_1:P(E,e_0)\to E$ is quotient, Corollary \ref{cpcequivalencetheorem} applies to give a homeomorphism $f:E\to E'$ such that $q\circ f=p$. It follows that $p$ is a covering projection.
\end{proof}

\begin{example}\label{counterexample}
By considering two one-dimensional planar sets, we construct a counterexample to the converse of (1) $\Rightarrow$ (2) in Theorem \ref{fibrationtheorem}. Define \[A=\{(x,-\sqrt{x-x^2})\in\bbr^2\mid x\in[0,1]\}\cup \bigcup_{n\in\bbn}\{(t,t/n)\in\bbr^2\mid 0\leq t\leq 2\}.\] Let $X_1=A\cup \left\{\left(x,\frac{1-x}{2}\right)\mid 1\leq x\leq 2\right\}$ and $X_2=A\cup [1,2]\times \{0\}$. Define $p:X_1\to X_2$ to be the identity on $A$ and $p\left(x,\frac{1-x}{2}\right)=(x,0)$, $1\leq x\leq 2$ on the additional line segment (See Figure \ref{counterexamplefig}). Notice that $p$ is a continuous bijection with the continuous path-covering property. However, $f$ is not a fibration since it does not have the homotopy lifting property with respect to the convergent sequence space $S=\{0\}\cup\{1/n\mid n\in\bbn\}$. In particular, we have $f:S\to X_1$ given by $f(0)=(1,0)$ and $f(1/n)=(1,1/n)$. If $g:S\times I\to X_2$ is defined by $g(0,t)=(0,t+1)$ and $g(1/n,t)=\left(t+1,\frac{t+1}{n}\right)$, then we have $p\circ f(s)=g(s,0)$; however, there is no continuous lift $\wt{g}:S\times I\to X_1$ such that $p\circ \wt{g}=g$.
\end{example}

\begin{figure}[H]\label{counterexamplefig}
\centering \includegraphics[height=1.6in]{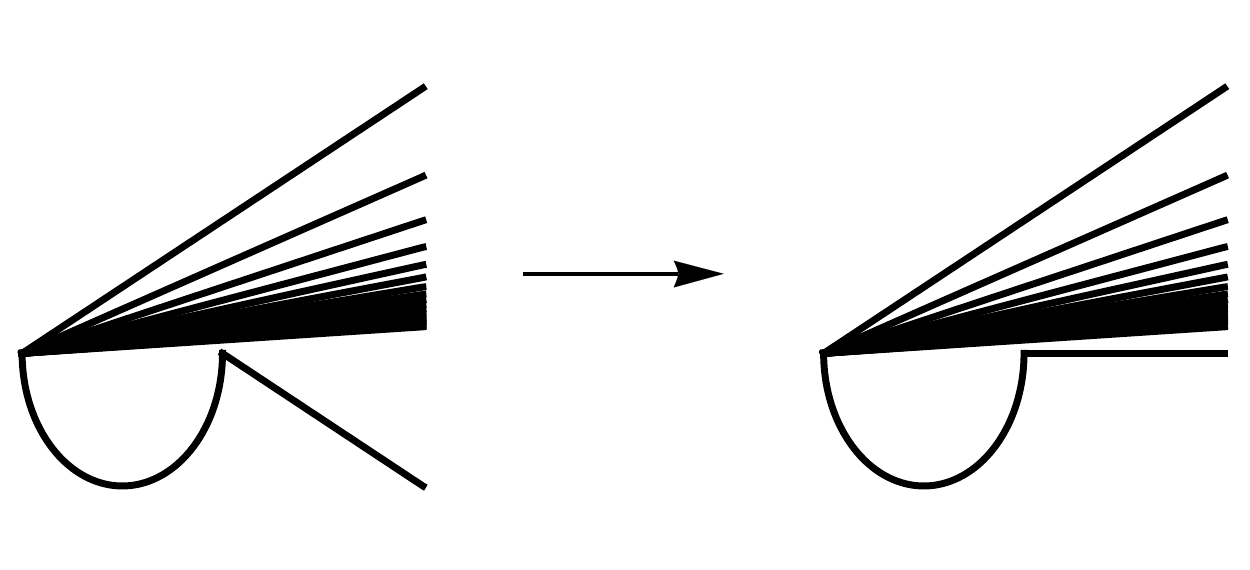}
\caption{The map $p:X_1\to X_2$ has the continuous path-covering property but is not a Hurewicz fibration.}
\end{figure}

\begin{example}\cite{FZ07}\label{heexample}
Consider the Hawaiian earring space \[\bbh=\bigcup_{n\in\bbn}\left\{(x,y)\in\bbr^2\mid \left(x-\frac{1}{n}\right)^2+y^2=\frac{1}{n^2}\right\}\] where $\ell_n:S^1\to \bbh$ denotes the standard counterclockwise loop around the n-th circle based at $b_0=(0,0)$. Although $\bbh$ does not admit a simply connected covering space (since it is not semilocally simply connected), it does admit a simply connected space $\wt{\bbh}$, which is a generalized universal covering space in the sense of \cite{FZ07}. The generalized covering map $p:\wt{\bbh}\to \bbh$ is characterized by its lifting property: given $\wt{x}\in\wt{\bbh}$ and a map $f:(Y,y)\to (\bbh,p(\wt{x}))$ from a path-connected, locally path-connected space $Y$, there exists a unique lift $\wt{f}:(Y,y)\to (\wt{\bbh},\wt{x})$ (satisfying $p\circ \wt{f}=f$) if and only if $f_{\#}(\pi_1(Y,y))=1$. It follows directly from this lifting criterion that $p$ is a Serre fibration with unique path-lifting. Moreover, as observed in \cite[Example 4.15]{FZ07}, distinct fibers of $p$ may not be homeomorphic. Hence, $p$ is not a Hurewicz fibration. We observe that, in fact, $p$ does not have the continuous path-covering property and, therefore, provides a counterexample to the converse of (2) $\Rightarrow$ (3) in Theorem \ref{fibrationtheorem}. 

As noted in \cite{FZ07}, $\wt{\bbh}=P(\bbh,b_0)/\mathord{\sim}$ may be constructed as the set of path-homotopy classes $[\alpha]$ of paths $\alpha\in P(\bbh,b_0)$. A basic open neighborhood of $[\alpha]$ is of the form $B([\alpha],U)=\{[\alpha\cdot\delta]\mid \delta(I)\subseteq U\}$ where $U$ is an open neighborhood of $\alpha(1)$. The map $p:\wt{\bbh}\to \bbh$ is endpoint projection $p([\alpha])=\alpha(1)$. In particular, if we take the class $[c_{b_0}]$ of the constant path as our basepoint in $\wt{\bbh}$, then the unique lift $\wt{\alpha}:(I,0)\to (\wt{\bbh},[c_{b_0}])$ ends at $[\alpha]$.

Notice that the sequence $\alpha_n=\ell_n\cdot\ell_1$ converges to $\ell_1$ in $P(\bbh,b_0)$. If $p$ had the continuous path-covering property, then the sequence of lifts $\wt{\alpha}_n:(I,0)\to (\wt{\bbh},b_0)$ would converge to the lift $\wt{\ell}_1$. In particular, the sequence of endpoints $\wt{\alpha}_n(1)=[\alpha_n]$ would converge to $\wt{\ell}_1(1)=[\ell_1]$ in $\wt{\bbh}$. However, if $U$ is any neighborhood of $b_0$ not containing the first circle of $\bbh$, then the basic open neighborhood $B([\ell_1],U)$ of $[\ell_1]$ only contains homotopy classes of paths that begin with $[\ell_1]$. Hence $[\alpha_n]\notin B([\ell_1],U)$ for any $n\in\bbn$; a contradiction. We conclude that $p$ does not have the continuous path-covering property.
\end{example}

Let $D^2$ denote the closed unit disk with basepoint $d_0=(1,0)$.

\begin{problem}[Dydak's Unique Lifting Problem]\cite[Problem 2.3]{Dydak}\label{dydaksproblem}
If $E$ is a connected, locally path-connected space and $p:E\to D^2$ is a map with the path-covering property, must $p$ be a homeomorphism?
\end{problem}

Dydak's Unique Lifting Problem is curiously difficult. We identify an interesting connection between this problem and Theorem \ref{fibrationtheorem}. To do so, we recall a functorial construction that ``locally path-connectifies" spaces. Given a space $X$ the \textit{locally path-connected coreflection} of $X$ is the space $lpc(X)$ with the same underlying set as $X$ but whose topology is generated by the basis consisting of all path components of the open sets in $X$. This new topology on $X$ is generally finer than the original topology, i.e. the identity function $id:lpc(X)\to X$ is continuous. Moreover $lpc(X)$ is characterized by its universal property: if $f:Y\to X$ is a map from a locally path-connected space $Y$, then $f:Y\to lpc(X)$ is continuous with respect to the topology of $lpc(X)$. An immediate consequence of this fact is the equality $X^Y=lpc(X)^Y$ of mapping sets if $Y$ is locally path connected. In particular, $X$ and $lpc(X)$ share the same set of paths and path-homotopies. Hence, $id:lpc(X)\to X$ is a bijective weak homotopy equivalence, which sometimes is even a Hurewicz fibration \cite[Example 2.4.8]{Spanier66}.

\begin{definition}
We say a map $p:E\to X$ has the \textit{disk-covering property} if for every $e\in E$, the induced function $F:(E,e)^{(D^2,d_0)}\to (X,p(e))^{(D^2,d_0)}$ is bijective.
\end{definition}

\begin{lemma}\label{diskcoveringlemma}
For any map $p:E\to X$, the following are equivalent.
\begin{enumerate}
\item $p$ has the disk-covering property,
\item $p$ has the path-covering property and the homotopy lifting property with respect to $I$.
\item $p:E\to X$ is a Serre fibration with totally path-disconnected fibers,
\item $p$ has totally path-disconnected fibers and the homotopy lifting property with respect to all first countable, locally path-connected, simply connected spaces.
\end{enumerate}
\end{lemma}

\begin{proof}
The directions (4) $\Rightarrow$ (3) $\Rightarrow$ (2) $\Rightarrow$ (1) are clear. Additionally, (1) $\Rightarrow$ (2) follows directly from the fact that $(I^2,I\times \{0\})$ has the homotopy extension property and $I^2/I\times \{0\}\cong D^2$. To prove (2) $\Rightarrow$ (4), suppose $p:E\to X$ has the path-covering property and the homotopy lifting property with respect to $I$. The fibers of $p$ are totally path disconnected by Proposition \ref{tpdfibersprop}. Note that Lemma \ref{basicliftingpropertieslemma} applies to $p$. Let $Z$ be a first countable, path-connected, locally path-connected space and $f:Z\to E$ and $g:Z\times I\to X$ be maps such that $p\circ f(z)=g(z,0)$ for all $z\in Z$. Fix $z_0\in Z$. Set $e_0=f(z_0)$ and $x_0=p(e_0)$. Since $p$ has unique lifting of all paths and path-homotopies and since $Z\times I$ is simply connected, there is a unique function $\wt{g}:(Z\times I,(z_0,0))\to (E,e_0)$ such that $p\circ \wt{g}=g$ defined so that if $(z,t)\in Z\times I $ and $\gamma$ is a path from $(z_0,0)$ to $(z,t)$, then $\wt{g}(z,t)$ is the endpoint of the lift $\wt{g\circ \gamma}\in P(E,e_0)$. It suffices to check that $\wt{g}$ is continuous. Consider a convergent sequence $\{(z_m,t_m)\}\to (z,t)$ in $Z\times I$. Since $Z\times I$ is first countable and locally path-connected, there is a path $\gamma\in P(Z\times I,(z_0,0))$ such that $\gamma(1/2)=(z,t)$ and if $s_m=\frac{1}{2}+\frac{1}{m+1}$, then $\gamma(s_m)=(z_m,t_m)$. Consider the lift $\wt{g\circ f}:(I,0)\to (E,e_0)$. Applying the uniqueness of path lifting and the well-definedness of $\wt{g}$, we see that $\wt{g\circ \gamma}(1/2)=\wt{g}(z,t)$ and $\wt{g\circ \gamma}(s_m)=\wt{g}(z_m,t_m)$ for all $m\in\bbn$. The continuity of $\wt{g\circ\gamma}$ applied to $\{t_m\}\to 1/2$ gives $\wt{g}(z_m,t_m)\to \wt{g}(z,t)$. Thus $\wt{g}$ is continuous.
\end{proof}

\begin{theorem}\label{dydaktheorem}
Dydak's Unique Lifting Problem has a positive answer if and only if properties (3) and (4) in Theorem \ref{fibrationtheorem} are equivalent for all maps.
\end{theorem}

\begin{proof}
First, suppose (3) and (4) in Theorem \ref{fibrationtheorem} are equivalent for all maps. If $E$ is a path-connected, locally path-connected space and $p:E\to D^2$ is a map with the path-covering property, then, by assumption, $p$ is a Serre fibration. Recall $d_0=(1,0)\in D^2$ and set $e_0\in p^{-1}(d_0)$. Let $\alpha:(I,0)\to (D^2,d_0)$, $\alpha(t)=(\cos(\pi t),-\sin(\pi t))$ be the arc on the boundary of the lower semicircle and $f:(I^2,\bfz)\to (D^2,d_0)$ be a homeomorphism such that $f(t,0)=\alpha(t)$. Since $p$ is a Serre fibration there exists a map $\wt{f}:(I^2,\bfz)\to (E,e_0)$ satisfying $p\circ \wt{f}=f$. Clearly, we have $p\circ (\wt{f}\circ f^{-1})=id_{D^2}$. 
\[\xymatrix{
& E \ar[d]^-{p}\\
I^2 \ar[ur]^-{\wt{f}} \ar[r]_-{f}^-{\cong} & D^2
}\]
Let $e\in E$ and find a path $\wt{\beta}:I\to E$ from $e_0$ to $e$. Let $\beta=p\circ\wt{\beta}$. Now $\wt{f}\circ f^{-1}\circ\beta:(I,0)\to (E,e_0)$ is a path satisfying $p\circ \wt{f}\circ f^{-1}\circ\beta=\beta$ and thus $\wt{f}\circ f^{-1}\circ\beta=\wt{\beta}$. In particular, $\wt{f}\circ f^{-1}\circ p(e)=\wt{f}\circ f^{-1}\circ \beta(1)=\wt{\beta}(1)=e$. Thus $\wt{f}\circ f^{-1}\circ p=id_E$, proving that $p$ is a homeomorphism with inverse $\wt{f}\circ f^{-1}$.

Next, we suppose Dydak's Unique Lifting Problem has a positive answer and that $p:E\to X$ is a map with the path-covering property. By Lemma \ref{diskcoveringlemma}, it suffices to show that $p$ has the disk-covering property. Let $e\in E$ and consider the induced map $F:(E,e)^{(D^2,d_0)}\to (X,p(e))^{(D^2,d_0)}$, $F(g)=p\circ g$. Since $p$ has unique path-lifting it is clear that $F$ is injective. Let $f:(D^2,d_0)\to (X,p(e_0))$ be a map. Let $D^2\times_{X}E=\{(b,e)\mid f(b)=p(e)\}$ be the pullback of $f$ and $p$ topologized as a subspace of $D^2\times E$. We take $Y$ to be the locally path-connected coreflection of the path component of $D^2\times_{X}E$ containing $(d_0,e_0)$. Let $q:Y\to D^2$ be the resulting projection. Since $p$ has the disk-covering property, the universal property of the pullback and the locally path-connected coreflection make it clear that $q$ has the path-covering property. Hence, by assumption, $q$ is a homeomorphism. It follows that if $i:Y\to D^2\times_{X}E$ is the canonical continuous inclusion (though it may not be an embedding) and $r:D^2\times_{X}E\to E$ is the projection, then $\wt{f}=r\circ i\circ q^{-1}:(D^2,d_0)\to (E,e_0)$ is a map such that $p\circ \wt{f}=f$. We conclude that $F$ is onto. Thus $p$ has the disk-covering property.
\[\xymatrix{
Y \ar[r]^-{i} \ar@<.5ex>[d]^{q} & D^2\times_{X}E \ar[r]^-{r} \ar[d] & E \ar[d]^-{p}\\
D^2 \ar@<.5ex>[u]^{q^{-1}} \ar@{=}[r] & D^2 \ar@{-->}[ur]^-{\wt{f}} \ar[r]_{f} & X
}\]
\end{proof}

If Dydak's problem has an affirmative answer, the following Corollary becomes a striking consequence; the conclusion implies that, within standard scenarios, only the unique lifting of paths is required to guarantee the strongest kind of structure, namely, that of a genuine covering projection. For instance, it would imply that any map $p:E\to X$ of connected manifolds with the path-covering property must be a covering projection.

\begin{corollary}\label{corollarytodydak}
Suppose Dydak's Unique Path-Lifting Problem has an affirmative answer. If $p:E\to X$ has the path-covering property where $X$ is first countable, locally path-connected, semilocally simply connected, and $ev_1:P(E,e_0)\to E$ is quotient, then $p$ is a covering projection.
\end{corollary}

\begin{proof}
Supposing the hypotheses, $p$ is a Serre fibration with totally path-disconnected fibers by Theorem \ref{dydaktheorem}. Let $H=p_{\#}(\pi_1(E,e_0))$. Given the conditions on $X$, there exists a covering projection $q:(E',e_0')\to (X,x_0)$ such that $q_{\#}(\pi_1(E',e_0'))=H$. Since $q$ has the continuous path-covering property, by Lemma \ref{liftinglemma}, there exists a unique continuous lift $\wt{p}:(E,e_0)\to (E',e_0')$ such that $q\circ \wt{p}=p$. Since both $p$ and $q$ have path-covering property and correspond to $H$, $\wt{p}$ is a bijection with the path-covering property.
\[\xymatrix{
P(E,e_0) \ar[d]_-{ev_1} \ar[r]^-{P(\wt{p})} & P(E',e_0') \ar[d]^-{ev_1}\\
E \ar[r]_-{\wt{p}} & E'
}\]
Since $X$ is first countable, so is the covering space $E'$. Let $\{e_m'\}\to e'$ be a convergent sequence in $E'$ and $\wt{p}(e_m)=e_m'$ and $\wt{p}(e)=e'$. Since $E'$ is first countable and locally path-connected, there is a path $\wt{\alpha}\in P(E',e_0')$ such that $\wt{\alpha}(1/2)=e'$ and $\wt{\alpha}(t_m)=e_m'$ where $t_m=\frac{1}{2}+\frac{1}{m+1}$, $m\in\bbn$. Since $\wt{p}$ has the path-covering property, there is a unique path $\wt{\beta}\in P(E,e_0)$ such that $\wt{p}\circ\wt{\beta}=\wt{\alpha}$. Since $\wt{\beta}(t_m)=e_m$ and $\wt{\beta}(1/2)=e$, the continuity of $\wt{\beta}$ gives $\{e_m\}\to e$. This proves that the inverse of $\wt{p}$ is continuous. Since $\wt{p}$ is a homeomorphism and $q$ is a covering map, $p$ is a covering map.
\end{proof}

\section{Inducing embeddings on topologized fundamental groups}\label{sectioninducedhomo}

Recall from the introduction the basic properties of quotient topology on the homotopy groups. A key feature of this topology is that a convergent net $\{\alpha_j\}_{j\in J}\to \alpha$ of based loops in $\Omega(X,x_0)$ gives rise to a convergent net of homotopy classes $\{[\alpha_j]\}_{j\in J}\to [\alpha]$ in $\pionex$.

\begin{lemma}\label{closedembeddingonloopsprop}
If $\{e\}$ is closed in $E$ and $p:E\to X$ has the continuous path-covering property, then the induced map $\Omega^n(p):\Omega^n(E,e)\to \Omega^n(X,p(e))$ is a closed embedding for $n=1$ and a homeomorphism for $n\geq 2$.
\end{lemma}
\begin{proof}
For the case $n=1$, note that since $\{e\}$ is closed in $E$, $ev_{1}^{-1}(e)=\Omega(E,e)$ is a closed subspace of $P(E,e)$. By assumption, $p$ induces a homeomorphism $P(p):P(E,e)\to P(X,p(e))$ on path spaces, which restricts to an embedding $\Omega(p):\Omega(E,e) \to \Omega(X,p(e))$. Recall that a path $\beta\in \Omega(X,p(e))$ lies in the image of $\Omega(p)$ if and only if the lift $\wt{\beta}\in P(E,e)$ ends at $e$, i.e. is a loop. Therefore, it suffices to show that $\Omega(p)$ has closed image in $\Omega(X,p(e))$. Suppose, to obtain a contradiction, that $\Omega(p)$ does not have closed image. Then there is a convergent net $\{\beta_j\}_{j\in J}\to \beta_{\infty}$ in $\Omega(X,p(e))$ where $\beta_j\in Im(\Omega(p))$ for every $j\in J$ and $\beta_{\infty}\notin Im(\Omega(p))$. This net uniquely defines a map $g:(F(J),v_0)\to (X,p(e))$ where $g(j,t)=\beta_j(t)$ for $j\in J\cup \{\infty\}$. According to Remark \ref{netstofansremark}, there is a unique map $\widetilde{g}:(F(J),v_0)\to (E,e)$ such that $p\circ \widetilde{g}=g$. Let $\wt{\beta}_{j}$ be the path $\wt{\beta}_j(t)=\wt{g}(j,t)$ for $j\in J\cup \{\infty\}$. Then $\{\wt{\beta}_j\}_{j\in J}\to \wt{\beta}_{\infty}$ in $P(E,e)$. But since $\beta_j$ lies in the image of $\Omega(p)$ for each $j\in J$, we have $\wt{\beta}_j\in \Omega(E,e)$ for all $j\in J$. Additionally, since $\beta_{\infty}$ does not lie in the image of $\Omega(p)$, $\beta_{\infty}$ does not lift to a loop, i.e. $\wt{\beta}_{\infty}\notin \Omega(E,e)$. However, this contradicts the fact that $\Omega(E,e)$ is closed in $P(E,e)$.

For $n\geq 2$, Corollary \ref{continuoushomotopycorollary} ensures that $p^{(I^n,\bfz)}:(E,e)^{(I^n,\bfz)}\to (X,p(e))^{(I^n,\bfz)}$ is a homeomorphism. Hence the restriction $\Omega^n(p):\Omega^n(E,e)\to\Omega^n(X,p(e))$ is an embedding. We check that $\Omega^n(p)$ is onto. Let $f:(I^n,\partial I^n)\to (X,p(e))$ be a map and consider the unique lift $\wt{f}:(I^n,\bfz)\to (E,e)$. It suffices to check that $\wt{f}(\partial I^n)=e$. Let $\mathbf{x}\in \partial I^n$ and $\gamma:(I,0)\to (I^n,\bfz)$ be the linear path from $\bfz$ to $\mathbf{x}$. Hence, we have a loop $f\circ \gamma:(I,\{0,1\})\to (X,p(e_2))$, which factors through the simply connected space $S^n\cong I^n/\partial I^n$. Since $[f\circ\gamma]=1\in \pi_1(X,p(e))$, Lemma \ref{basicliftingpropertieslemma} guarantees that the unique lift $\wt{f\circ\gamma}\in P(E,e)$ is a loop. Since $(p\circ \wt{f})\circ\gamma=f\circ\gamma=p\circ \wt{f\circ\gamma}$, unique path-lifting gives $\wt{f}\circ\gamma=\wt{f\circ\gamma}$. Hence $\wt{f}(\mathbf{x})=\wt{f}(\gamma(1))=\wt{f\circ\gamma}(1)=e$. This proves $\wt{f}(\partial I^n)=e$.
\end{proof}
\begin{theorem}\label{closedembeddingonpi1}
If $\{e\}$ is closed in $E$ and $p:E\to X$ has the continuous path-covering property, then the induced homomorphism $p_{\#}:\pi_n(E,e)\to \pi_n(X,p(e))$ is a closed embedding for $n=1$ and an isomorphism of quasitopological groups for $n\geq 2$.
\end{theorem}
\begin{proof}
For $n=1$, $p_{\#}$ is injective by Lemma \ref{basicliftingpropertieslemma} and continuous by the functorality of the quotient topology. Consider the following commutative diagram where the vertical maps are the natural quotient maps identifying homotopy classes.
\[\xymatrix{
\Omega(E,e) \ar[d]_-{q_E} \ar[r]^-{P(p)} & \Omega(X,p(e)) \ar[d]^-{q_X}\\
\pi_1(E,e) \ar[r]_-{p_{\#}} & \pi_1(X,p(e))
}\]
Suppose that $C\subseteq \pi_1(E,e)$ is closed and non-empty. Then $q_{E}^{-1}(C)$ is closed in $\Omega(E,e)$ and since the top map is a closed embedding by Lemma \ref{closedembeddingonloopsprop}, $P(p)(q_{E}^{-1}(C))$ is closed in $\Omega(X,p(e))$. Since $q_{X}$ is a quotient map, it suffices to show that $q_{X}^{-1}(p_{\#}(C))$ is closed. However, this is a consequence of the equality $P(p)(q_{E}^{-1}(C))=q_{X}^{-1}(p_{\#}(C))$, which may be verified using standard unique lifting arguments.

For $n\geq 2$, consider the following commuting diagram where the vertical maps are the natural quotient maps. By Lemma \ref{closedembeddingonloopsprop}, the top map is a homeomorphism.
\[\xymatrix{
\Omega^n(E,e) \ar[d]_-{q_E} \ar[r]^-{\Omega^n(p)} & \Omega^n(X,p(e)) \ar[d]_{q_X} \\
\pi_n(E,e) \ar[r]_-{p_{\#}} & \pi_n(X,p(e))
}\]
The universal property of quotient maps ensures that $p_{\#}$ is a topological quotient map. By Theorem \ref{fibrationtheorem}, $p$ is a Serre fibration with totally path-disconnected fibers. The injectivity of $p_{\#}$ follows from the long exact sequence on homotopy groups associated with $p$. Thus, $p_{\#}$ is a group isomorphism and a homeomorphism.
\end{proof}

\begin{corollary}
Suppose $p_i:E_i\to X$, $i\in\{1,2\}$ are maps with the continuous path-covering property and $f:E_1\to E_2$ is a weak homotopy equivalence such that $p_2\circ f=p_1$. Then $f$ is a weak topological homotopy equivalence. 
\end{corollary}

\begin{proof}
First, notice that $f$ must be a bijection. By Lemma \ref{compositionlemma}, $f$ has the continuous path-covering property and it follows from Theorem \ref{closedembeddingonpi1} that $f$ is a weak topological homotopy equivalence.
\end{proof}

\begin{theorem}\label{liftingcorrespondencetheorem}
$p:(E,e_0)\to (X,x_0)$ has the continuous path-covering property and $H=p_{\#}(\pi_1(E,e_0))$ where $\{x_0\}$ is closed in $X$. Then there is a canonical continuous bijection $\phi:\pionex/H\to p^{-1}(x_0)$ defined by $\phi(H[\alpha])=\wt{\alpha}(1)$ where $\wt{\alpha}$ is the unique lift of $\alpha$ starting at $e_0$. Moreover, $\phi$ is a homeomorphism if and only if $ev_1:P(E,e_0)\to E$ is a quotient map.
\end{theorem}

\begin{proof}
The map $\phi$ is analogous to the correspondence used in classical covering space theory and is well-defined by applying Lemma \ref{basicliftingpropertieslemma}. To verify the continuity of $\phi$, we consider the following diagram where $i$ is inclusion and the vertical map $q(\alpha)=H[\alpha]$ is quotient.
\[\xymatrix{
\Omega(X,x_0) \ar[ddrr]^{\psi} \ar[dd]_-{q} \ar[rr]^{i} && P(X,x_0) \ar[d]_{\cong}^-{P(p)^{-1}} \\
   &  & P(E,e_0) \ar[d]^-{ev_1}\\
\pionex/H \ar@{-->}[rr]_-{\phi} && E
}\]
The composition $\psi=ev_1\circ P(p)^{-1}\circ i$ is continuous, which maps a loop $\alpha$ based at $x$ to the endpoint $\wt{\alpha}(1)$ of the lift starting at $e_0$. Since $E$ is assumed to be path connected, $\psi$ maps $\Omega(X,x_0)$ onto the fiber $p^{-1}(x_0)$. Lemma \ref{basicliftingpropertieslemma} implies that $\psi(\alpha)=\psi(\beta)$ if and only if $H[\alpha]=H[\beta]$. Since $q$ is quotient, we have a unique continuous bijection $\phi:\pionex/H\to p^{-1}(x_0)$ defined by $\phi(H[\alpha])=\psi(\alpha)$.

Moreover, if $\phi$ is quotient, then so is $\psi=ev_1\circ (P(p)^{-1}\circ i)$ from which it follows that $ev_1$ is quotient. Conversely, if $ev_1$ is quotient, then $P(p)^{-1}\circ ev_1$ is quotient. Since $\{x_0\}$ is closed, $p^{-1}(x_0)$ is closed in $E$. Since $(P(p)^{-1}\circ ev_1)^{-1}(p^{-1}(x_0))=\Omega(X,x_0)$, the restriction $\psi:\Omega(X,x_0)\to p^{-1}(x_0)$ of the quotient map $P(p)^{-1}\circ ev_1$ is also quotient. Since $\phi\circ q=\psi$ where $q$ and $\psi$ are quotient, $\phi$ must also be quotient and thus a homeomorphism.
\end{proof}

\begin{corollary}\label{tpdcorollary}
If $p:(E,e_0)\to (X,x_0)$ has the continuous path-covering property and $H=p_{\#}(\pi_1(E,e_0))$, then the coset space $\pionex/H$ is totally path disconnected.
\end{corollary}

\begin{proof}
By Proposition \ref{tpdfibersprop}, $p^{-1}(x_0)$ is totally path disconnected and Theorem \ref{liftingcorrespondencetheorem} implies that $\pionex/H$ continuously injects into $p^{-1}(x_0)$.
\end{proof}

\section{Classifying maps with continuous path-covering property}\label{sectionclassification}

To make the statement of Theorem \ref{mainresult1} precise, we define the following relations.

\begin{definition}
Consider two maps $p_1:E_1\to X$ and $p_2:E_2\to X$ with the continuous path-covering property.
\begin{enumerate}
\item We say $p_1$ and $p_2$ are \textit{equivalent} if there exists a homeomorphism $h:E_1\to E_2$ such that $p_2\circ h=p_1$ and we refer to $h$ as an \textit{equivalence}.
\item A \textit{simple weak equivalence between $p_1$ and $p_2$} is a triple $(p_3,f_1,f_2)$ where $p_3:E_3\to X$ has the continuous path-covering property and $f_i:E_3\to E_i$, $i\in\{1,2\}$ are weak homotopy equivalences such that $p_i\circ f_i=p_3$. If a simple weak equivalence exists between $p_1$ and $p_2$, we write $p_1\sim_{s}p_2$.
\[\xymatrix{
E_1 \ar[dr]_-{p_1} & E_3  \ar[l]_-{f_1} \ar[d]^-{p_3} \ar[r]^-{f_2} & E_2 \ar[dl]^-{p_2} \\
& X
}\]
\item We say $p_1$ and $p_2$ are \textit{weakly equivalent} if there exists a finite chain of simple weak equivalences $p_1=q_1\sim_{s} q_2\sim_{s} \cdots \sim_{s}q_m=p_2$. If such a chain exists, we say $p_1$ and $p_2$ are \textit{weakly equivalent}.
\end{enumerate}
\end{definition}

\begin{remark}\label{surjectiveremark}
Note that if we have $p_2\circ f=p_1$ for maps $p_i:E_i\to X$ with the continuous path-covering property, then $f$ has the continuous path-covering property by Lemma \ref{compositionlemma} and $f$ must be surjective. If, in addition, $f$ is $\pi_1$-surjective, then standard lifting arguments give that $f$ is a bijection. If $f$ induces an isomorphism on $\pi_1$, then $f$ must also be a weak topological homotopy equivalence by Theorem \ref{closedembeddingonpi1} (assuming $X$ is at least $T_1$). Hence, one may equivalently define ``simple weak equivalence" by either replacing weak homotopy equivalences $f_1,f_2$ with the weaker notion of $\pi_1$-surjective maps or the stronger notion of bijective, weak topological homotopy equivalences.
\end{remark}

\begin{lemma}\label{weakequivalencelemma}
If $p_1:E_1\to X$ and $p_2:E_2\to X$ are weakly equivalent maps with the continuous path-covering property and $e_1\in E_1$, then $(p_1)_{\#}(\pi_1(E_1,e_1))=(p_2)_{\#}(\pi_1(E_2,e_2))$ for some $e_2\in E_2$.
\end{lemma}

\begin{proof}
Fix $e_1\in E_1$ and $x_0=p_1(e_1)$. If there is a weak homotopy equivalence $f:E_1\to E_2$ such that $p_2\circ f=p_1$, we set $e_2=f(e_1)$. Since $f$ is $\pi_1$-surjective, the equality $(p_1)_{\#}(\pi_1(E_1,e_1))=(p_2)_{\#}(\pi_1(E_2,e_2))$ follows. If there is a weak homotopy equivalence $g:E_2\to E_1$ such that $p_1\circ g=p_2$, then $g$ is surjective by Remark \ref{surjectiveremark}. Hence, we may find $e_2\in E_2$ with $g(e_2)=e_1$ from which $(p_1)_{\#}(\pi_1(E_1,e_1))=(p_2)_{\#}(\pi_1(E_2,e_2))$ follows. Applying zig-zags of simple weak equivalences, the lemma follows. 
\end{proof}

Notice that Lemma \ref{weakequivalencelemma} implies that the weak equivalence classes of maps with the continuous path-covering property over a given space $X$ form a set. To give a self-contained proof of Theorem \ref{mainresult1}, we require a sequence of lemmas involving quotient space constructions. 

\begin{lemma}\label{coreflectionlemma}
Suppose $X$ is a path-connected Hausdorff space, $x_0\in X$, and $c(X)$ is the space with the same underlying set as $X$ but with the quotient topology inherited from $ev_1:P(X,x_0)\to X$. Then
\begin{enumerate}
\item the identity function $f:c(X)\to X$ has the continuous path-covering property and is a weak topological homotopy equivalence,
\item $ev_1:P(c(X),x_0)\to c(X)$ is quotient, i.e. $c(c(X))=c(X)$.
\end{enumerate}
\end{lemma}

\begin{proof}
Let $J$ be a directed set and $g:(F(J),v_0)\to (X,x_0)$ be a based map on the directed arc-fan over $J$. Applying Remark \ref{netstofansremark}, we show that $g:(F(J),v_0)\to (c(X),x_0)$ is continuous. Since $F(J)$ is contractible, $ev_1:P(F(J),v_0)\to F(J)$ is a retraction and is therefore a topological quotient map. Consider the induced map $P(g)$ in the diagram below.
\[\xymatrix{
P(F(J),v_0) \ar[d]_-{ev_1} \ar[r]^{P(g)} & P(X,x_0) \ar[d]^-{ev_1}\\
F(J) \ar[r]_-{g} & c(X)
}\]
Since the top composition is continuous and the left vertical map is quotient, the bottom map is continuous by the universal property of quotient maps. Therefore, since all maps of directed arc-fans lift, $f$ has the continuous path-covering property. Having a finer topology than $X$, $c(X)$ is Hausdorff. Lemma \ref{closedembeddingonpi1} then gives that $f$ induces a closed embedding on $\pi_1$ and a topological isomorphism on $\pi_n$ for $n\geq 2$. Hence, it suffices to show $f$ is $\pi_1$-surjective. Given a loop $\alpha\in \Omega(X,x_0)$, there is a unique lift $\wt{\alpha}\in P(c(X),x_0)$ such that $f\circ \wt{\alpha}=\alpha$. However, since the underlying function of $f$ is the identity, it must be that $\wt{\alpha}=\alpha$ as a loop. Since $\Omega(f):\Omega(c(X),x_0)\to \Omega(X,x_0)$ is a bijection, it follows that $f_{\#}:\pi_1(c(X),x_0)\to \pi_1(X,x_0)$ is surjective. This completes the proof of (1).

For (2), recall that we have shown $f$ has the continuous path-covering property. Hence, the top map in the triangle below is a homeomorphism and the right map is quotient by construction, the left evaluation map is the composition of quotient maps and is therefore quotient.
\[\xymatrix{
P(c(X),x_0)  \ar[dr]_-{ev_1} \ar[rr]^-{P(f)} && P(X,x_0) \ar[dl]^-{ev_1} \\
& c(X)
}\]
\end{proof}

\begin{example}\label{zeemansexample}
Path-connected spaces $(Z,z_0)$ for which $c(Z)\to Z$ is not a homeomorphism exist, e.g. spaces $X_1$ and $X_2$ in Example \ref{counterexample} and Zeeman's example \cite[Example 6.6.14]{HW60}. For any such space, the identity function $c(Z)\to Z$ and the identity map $Z\to Z$ are non-equivalent, weakly equivalent maps with the continuous path-covering property that both correspond to $H=\pi_1(Z,z_0)$. Therefore, the main statement of Theorem \ref{mainresult1} only holds using our notion of weak equivalence.
\end{example}

We use the construction of $c(X)$ to prove the converse of Lemma \ref{weakequivalencelemma}.

\begin{lemma}\label{uniquenesslemma}
Let $p_1:E_1\to X$ and $p_2:E_2\to X$ be maps with the continuous path-covering property. Then the following are equivalent:
\begin{enumerate}
\item $p_1$ and $p_2$ are weakly equivalent,
\item for every $e_1\in E_1$, we have $(p_1)_{\#}(\pi_1(E_1,e_1))=(p_2)_{\#}(\pi_1(E_2,e_2))$ for some $e_2\in E_2$. 
\item for all $e_1\in E_1$ and $e_2\in E_2$ with $p_1(e_1)=x_0=p_2(e_2)$, the subgroups $(p_1)_{\#}(\pi_1(E_1,e_1))$ and $(p_2)_{\#}(\pi_1(E_2,e_2))$ are conjugate in $\pionex$.
\end{enumerate}
Moreover, if $ev_1:P(E_1,e_1)\to E_1$ and $ev_1:P(E_2,e_2)\to E_2$ are quotient, then ``weak equivalence" may be replaced by ``equivalence."
\end{lemma}
\begin{proof}
(1) $\Rightarrow$ (2) is Lemma \ref{weakequivalencelemma} and (2) $\Leftrightarrow$ (3) follows from the standard covering space theory arguments. To prove (2) $\Rightarrow$ (1), suppose that $(p_1)_{\#}(\pi_1(E_1,e_1))=(p_2)_{\#}(\pi_1(E_2,e_2))$ for some $e_1\in E_1$ and $e_2\in E_2$. Let $c(E_1)$ and $c(E_2)$ be the spaces constructed as in Lemma \ref{coreflectionlemma} so that the continuous identity functions $f_1:c(E_1)\to E_1$ and $f_2:c(E_2)\to E_2$ have the continuous path-covering property and are weak topological homotopy equivalences. For $i\in \{1,2\}$, let $q_i:c(E_i)\to X$ be the map $q_i=p_i\circ f_i$. Since $q_i$ is the composition of maps with the continuous path-covering property, $q_i$ also has the continuous path-covering property (recall Lemma \ref{compositionlemma}). Moreover, since $f_1$ and $f_2$ are weak homotopy equivalences, we have $(q_1)_{\#}(\pi_1(c(E_1),e_1))=(q_2)_{\#}(\pi_1(c(E_2),e_2))$. By (2) of Lemma \ref{coreflectionlemma}, $ev_1:P(c(E_i),e_i)\to c(E_i)$ is quotient for $i\in \{1,2\}$. Therefore, Lemma \ref{cpcequivalencetheorem} applies to give a homeomorphism $h:(c(E_1),e_1)\to (c(E_2),e_2)$ such that $q_2\circ h=q_1$.
\[\xymatrix{
E_1 \ar[drr]_{p_1}  & c(E_1) \ar[l]_-{f_1} \ar[dr]^-{q_1} \ar[rr]^-{h}_{\cong}  &&  c(E_2) \ar[r]^-{f_2} \ar[dl]_-{q_2}  & E_2 \ar[dll]^-{p_2} \\
&& X
}\]
For the final statement of the lemma, suppose $ev_1:P(E_i,e_i)\to E_i$ is quotient for $i\in\{1,2\}$. Then $f_1$ and $f_2$ are true identity maps and thus homeomorphisms. It follows that $p_1$ and $p_2$ are equivalent.
\end{proof}
The previous lemma settles uniqueness claims in Theorem \ref{mainresult1}. We now focus on existence. Fix based space $(X,x_0)$, a subgroup $H\leq \pionex$, and let $\tXh=P(X,x_0)/\mathord{\sim}$ be the quotient space where $\alpha\sim \beta$ if and only if $\alpha(1)=\beta(1)$ and $[\alpha\cdot\beta^{-}]\in H$. Let $H[\alpha]$ denote the equivalence class of $\alpha\in P(X,x_0)$ and $q_H:P(X,x_0)\to \tXh$, $q_H(\alpha)=H[\alpha]$ be the quotient map. We write $\txh$ to represent $H[c_{x_0}]$, which we take to be the basepoint of $\tXh$. Let $p_H:\tXh\to X$, $p_H(H[\alpha])=\alpha(1)$ be the endpoint evaluation map.

\begin{lemma}\label{retractionlemma}
For any path-connected space $(X,x_0)$,
\begin{enumerate}
\item $ev_1:P(\tXh,\txh)\to \tXh$ is a quotient map,
\item $P(p_H):P(\tXh,\txh)\to P(X,x_0)$ is a retraction.
\end{enumerate}
\end{lemma}

\begin{proof}
Consider the following commutative diagram.
\[\xymatrix{
P(P(X,x_0),c_{x_0}) \ar@/^2pc/[rr]^-{P(ev_1)} \ar[d]_-{ev_1} \ar[r]^-{P(q_H)} &  P(\tXh,\txh) \ar[d]_-{ev_1} \ar[r]^-{P(p_H)} & P(X,x_0) \ar[d]^-{ev_1}\\
P(X,x_0) \ar@/_2pc/[rr]_-{ev_1} \ar[r]_-{q_H} & \tXh \ar[r]_-{p_H} & X
}\]
Since $P(X,x_0)$ is contractible in a canonical way, we may define a map $\mathscr{S}:P(X,x_0)\to P(P(X,x_0),c_{x_0})$ by setting $\mathscr{S}(\alpha)(t)(s)=\alpha(st)$.
\begin{itemize}
\item For the top map $P(ev_1)$, notice that $P(ev_1)(\beta)=ev_1\circ\beta\in P(X,x_0)$ and thus $(P(ev_1)(\beta))(t)=\beta(t)(1)$. Therefore, \[P(ev_1)(\mathscr{S}(\alpha))(t)=\mathscr{S}(\alpha)(t)(1)=\alpha(t)\]
for all $t\in I$, giving $P(ev_1)\circ \mathscr{S}=id_{P(X,x_0)}$.
\item For the left vertical map $ev_1:P(P(X,x_0),c_{x_0})\to \pxxo$, we have $\mathscr{S}(\alpha)(1)=\alpha$ and thus $ev_1\circ \mathscr{S}=id_{P(X,x_0)}$.
\end{itemize}
Hence, $\scrs$ is a section to both the top map $P(ev_1)$ and left vertical map $ev_1$. In particular, both maps are quotient. In the left square, the composition $q_H\circ ev_1$ is quotient. It follows that the middle vertical map $ev_1:P(\tXh,\txh)\to \tXh$ is quotient, proving (1). In the top triangle, we have $P(p_H)\circ (P(q_H)\circ \mathscr{S})=P(ev_1)\circ \mathscr{S}=id_{P(X,x_0)}$ and thus $P(q_H)\circ \mathscr{S}$ is a section to $P(p_H)$, proving (2).
\end{proof}

\begin{remark}
The section $P(q_H)\circ \mathscr{S}$ in the proof of Lemma \ref{retractionlemma} guarantees that for given $H\leq \pionex$, every path $\alpha\in P(X,x_0)$ admits a canonical lift $P(q_H)\circ \mathscr{S}(\alpha)=\wt{\alpha}_{H}:I\to \tXh$ of $\alpha$ called the \textit{standard lift} and defined by $\wt{\alpha}_{H}(t)=H[\alpha_t]$ where $\alpha_t(s)=\alpha(st)$ is the linear reparameterization of $\alpha|_{[0,t]}$ and $\alpha_0=c_{x_0}$.
\end{remark}

\begin{proposition}
The endpoint projection $p_H:\tXh\to X$ is a quotient map if and only if $ev_1:P(X,x_0)\to X$ is quotient.
\end{proposition}

\begin{proof}
Since $ev_1=p_H\circ q_H$ as maps $\pxxo\to X$ where $q_H$ is quotient, the conclusion follows from basic properties of quotient maps.
\end{proof}

For a given path $\alpha\in P(X,x_0)$, consider each of the following pullbacks (with the respective subspace topology):
\begin{itemize}
\item $G_{\alpha}=\{(x,t)\in X\times I\mid \alpha(t)=x\}$ is the graph of $\alpha$,
\item $E_{\alpha,H}=\{(H[\beta],t)\in\tXh\times I\mid \beta(1)=\alpha(t)\}$,
\item $P_{\alpha}=\{(\beta,t)\in \pxxo\times I\mid \beta(1)=\alpha(t)\}$.
\end{itemize}
Recalling that $\alpha_t(s)=\alpha(st)$ for $t\in I$, define $\phi_{\alpha}:P_{\alpha}\to \Omega(X,x_0)$ by $\phi_{\alpha}(\beta,t)=\beta\cdot\alpha_{t}^{-}$. Since $t\mapsto \alpha_t$ defines a path in $P(X,x_0)$ and concatenation $\{(\gamma,\delta)\in P(X,x_0)^2\mid \gamma(1)=\delta(1)\}\to \Omega(X,x_0)$, $(\gamma,\delta)\mapsto \gamma\cdot\delta^{-}$ is continuous, $\phi_{\alpha}$ is continuous.

\begin{lemma}\label{psicontinuitylemma}
If a path $\alpha\in\pxxo$ has closed graph $G_{\alpha}\subseteq X\times I$, then the function $\psi_{\alpha,H}:E_{\alpha,H}\to \pionex/H$, $\psi_{\alpha,H}(H[\beta],t)=H[\beta\cdot\alpha_{t}^{-}]$ is continuous.
\end{lemma}

\begin{proof}
First, we observe that $\psi_{\alpha,H}$ is well-defined. Indeed, if $(H[\beta],t)=(H[\gamma],t)$, then $[\beta\cdot\alpha_{t}^{-}][\alpha_{t}\cdot\gamma^{-}]=[\beta\cdot\gamma^{-}]\in H$ and thus $H[\beta\cdot\alpha_{t}^{-}]=H[\gamma\cdot\alpha_{t}^{-}]$.

Let $q_H:\pxxo\to \tXh$ and $\pi_H:\Omega(X,x_0)\to \pionex/H$ denote the canonical quotient maps. Since $I$ is locally compact Hausdorff, the product map $q_H\times id:\pxxo\times I\to \tXh\times I$ is a quotient map \cite{whiteheadquotient}. Since the graph $G_{\alpha}$ is assumed to be closed, the set $E_{\alpha,H}=(p_H\times id)^{-1}(G_{\alpha})$ is closed in $\tXh\times I$. We have $P_{\alpha}=(q_H\times id)^{-1}(E_{\alpha,H})$. Therefore, the restriction $(q_H\times id)|_{P_{\alpha}}:P_{\alpha}\to E_{\alpha,H}$ is a quotient map. Since $\psi_{\alpha,H}\circ (q_H\times id)|_{P_{\alpha}}=\pi_H\circ \phi_{\alpha}$ is continuous and $(q_H\times id)|_{P_{\alpha}}$ is quotient, $\psi_{\alpha,H}$ is continuous.
\[\xymatrix{
\pxxo\times I \ar[d]_-{q_H\times id} & P_{\alpha } \ar@{_{(}->}[l] \ar[d]_{(q_H\times id)|_{P_{\alpha}}} \ar[rr]^-{\phi_{\alpha}} && \Omega(X,x_0) \ar[d]^{\pi_H}\\
\tXh\times I \ar[d]_-{p_H\times id} & E_{\alpha,H} \ar[d] \ar@{_{(}->}[l]  \ar[rr]^-{\psi_{\alpha,H}} && \pionex/H \\
X\times I & G_{\alpha} \ar@{_{(}->}[l]
}\]
\end{proof}

\begin{theorem}\label{tpdtheorem}
If the graph $G_{\alpha}$ of every path $\alpha\in\pxxo$ is closed in $X\times I$ (e.g. if $X$ is Hausdorff) and $\pionex/H$ is totally path disconnected, then $p_H:\tXh\to X$ has the continuous path-covering property and $(p_H)_{\#}(\pi_1(\tXh,\txh))=H$.
\end{theorem}

\begin{proof}
Suppose $X$ satisfies the hypotheses of the theorem. Recall that $P(p_H):P(\tXh,\txh)\to P(X,x_0)$ is a topological retraction by Lemma \ref{retractionlemma}. Therefore, it is enough to show that $P(p_H)$ is injective, i.e. that $p_H:\tXh\to X$ has the unique path-lifting property. Fix a path $\alpha\in\pxxo$. It suffices to show that the standard lift $\wt{\alpha}_H\in P( \tXh,\txh)$ given by $\wt{\alpha}_{H}(t)=H[\alpha_t]$ is the \textit{only} lift of $\alpha$ starting at $\txh$.

Suppose, to obtain a contradiction, that $\wt{\beta}:I\to \tXh$ is a lift of $\alpha$ starting at $\txh$ such that $\wt{\beta}\neq \wt{\alpha}_H$. By restricting the domain if necessary, we may assume that $\wt{\beta}(1)\neq \wt{\alpha}_H(1)$. Write $\wt{\beta}(t)=H[\beta_t]$ for paths $\beta_t:(I,0,1)\to (X,x_0,\alpha(t))$. From the definition of the standard lift, $\wt{\beta}(1)\neq \wt{\alpha}_H(1)$ implies that $[\beta_1\cdot\alpha^{-}]\notin H$. Since $p_H\circ\wt{\beta}(t)=\beta_t(1)=\alpha(t)$ for all $t\in I$, there is a well-defined path $\wt{\beta}':I\to E_{\alpha,H}$ given by $\wt{\beta}'(t)=(\wt{\beta}(t),t)$. By Lemma \ref{psicontinuitylemma}, $\psi_{\alpha,H}:E_{\alpha,H}\to\pionex/H$ is continuous. Therefore, $\psi_{\alpha,H}\circ\wt{\beta}':I\to\pionex/H$ is a continuous path from $\psi_{\alpha,H}\circ\wt{\beta}'(0)=\psi_{\alpha,H}(\txh,0)=H[c_{x_0}\cdot\alpha_{0}^{-}]=\txh$ to $\psi_{\alpha,H}\circ\wt{\beta}'(1)=\psi_{\alpha,H}(H[\beta_1],1)=H[\beta_1\cdot\alpha^{-}]$. However, $H[\beta_1\cdot\alpha^{-}]\neq \txh$, showing that $\psi_{\alpha,H}\circ\wt{\beta}'$ is a non-constant path in $\pionex/H$; a contradiction of the assumption that $\pionex/H$ is totally path disconnected. We conclude that $p_H$ has the continuous path-covering property.

Finally, note that if $\alpha\in \Omega (X,x_0)$, then the standard lift $\wt{\alpha}_H$ is a loop $\Leftrightarrow$ $\wt{\alpha}_H(1)=\txh$ $\Leftrightarrow$ $H[\alpha]=H[c_{x_0}]$ $\Leftrightarrow$ $[\alpha]\in H$. It follows from (2) of Lemma \ref{basicliftingpropertieslemma} that $(p_H)_{\#}(\pi_1(\tXh,\txh))=H$.
\end{proof}

\begin{proof}[Proof of Theorem \ref{mainresult1}]
We note that the uniqueness (up to weak equivalence) conditions in Theorem \ref{mainresult1} are guaranteed by Lemma \ref{uniquenesslemma}. If there exists a map $p:(E,e)\to (X,p(e))$ with the continuous path-covering property such that $H=p_{\#}(\pi_1(E,e))$, then $\pionex/H$ is totally path-disconnected by Corollary \ref{tpdcorollary}. Conversely, if $\pionex/H$ is totally path-disconnected, then $p_H:\tXh\to X$ has the continuous path-covering property and satisfies $(p_H)_{\#}(\pi_1(\tXh,\txh))=H$ by Theorem \ref{tpdtheorem}. This proves the main statement of Theorem \ref{mainresult1}.

Part (1) follows by combining the main statement with Corollary \ref{cpcequivalencetheorem}. For Part (2), recall that $ev_1:P(\tXh,\txh)\to \tXh$ is quotient by (1) of Lemma \ref{retractionlemma}. Since every weak equivalence class of maps $E\to X$ with the continuous path-covering property is represented by a map of the form $p_H$, Part (2) follows.
\end{proof}
The existence portion of Theorem \ref{mainresult1} indicates that maps $p:E\to X$ with the continuous path-covering property where $E$ is path-connected exist very often. The next corollary is the case $H=1$ of Theorem \ref{mainresult1}.
\begin{corollary}\label{univcovcorollary}
If $X$ is Hausdorff, then there exists a smiply connected space $E$ and map $p:E\to X$ with the continuous path-covering property if and only if $\pionex$ is totally path-disconnected.
\end{corollary}

\begin{example}
Consider the canonical homomorphism $\Psi:\pionex\to \check{\pi}_1(X,x_0)$ to the first shape homotopy group (see \cite[Section 3]{FZ07}). The first shape homotopy group $\check{\pi}_1(X,x_0)$ is naturally an inverse limit of discrete groups and with the inverse limit topology on $\check{\pi}_1(X,x_0)$, the natural map $\Psi$ is continuous \cite[p. 79]{BFqtpfg}. If $\Psi$ is injective, then $\pionex$ continuously injects into an inverse limit of discrete spaces and is therefore totally path disconnected. By Corollary \ref{univcovcorollary}, $X$ must admit a simply connected space $E$ and map $p:E\to X$ with the continuous path-covering property. Spaces for which $\Psi$ is injective include, but are not limited to, all one-dimensional spaces \cite{CC06,Eda98}, planar sets \cite{FZ05}, and certain trees of manifolds \cite{FG05}.
\end{example}
\section{A remark on topological structure}\label{sectionstructure}
As noted in the introduction, it is not possible to characterize fibrations with unique path-lifting or other maps defined abstractly in terms of unique lifting properties \textit{up to homeomorphism} using the (topologized) fundamental group. However, there are many situations, where one can choose a highly structure representative map $p:E\to X$ from a given weak equivalence class. In this section, we focus on the locally path-connected case. To simplify our terminology, we say that a based map $p:(E,e_0)\to (X,x_0)$ with the continuous path-covering property \textit{corresponds} to a subgroup $H\leq \pionex$ if $p_{\#}(\pi_1(E,e_0))=H$. In what follows we will implicitly use the fact from Theorem \ref{mainresult1} that maps $p_1:E_1\to X$ and $p_2:E_2\to X$ with the continuous path-covering property are weakly equivalent if and only if for every $e_1\in E_1$, there exists $e_2\in E_2$ such that $p_1(e_1)=p_2(e_2)$ and $(p_1)_{\#}(\pi_1(E_1,e_1))=(p_2)_{\#}(\pi_1(E_2,e_2))$.

Inverse limits of path-connected spaces are not always path connected. Since we only wish to consider non-empty, path-connected domains, the next definition will simplify the exposition to follow.

\begin{definition}\label{inverselimitsdef}
Fix a class $\mathscr{C}$ of maps with the continuous path-covering property. Suppose $J$ is a directed set, $p_j\in\mathscr{C}$ for all $j\in J$, and $f_{j,j'}:E_j\to E_j'$ are maps satisfying $p_j\circ f_{j,j'}=p_{j'}$ whenever $j\geq j'$ in $J$. Suppose $E$ is a non-empty path component of the inverse limit $\varprojlim_{j}(E_j,f_{f,f'})$ and let $p:E\to X$ be the restriction of $\varprojlim_{j}p_j:\varprojlim_{j}(E_j,f_{f,f'})\to X$. We refer to the map $p$ as \textit{an inverse limit of maps of type }$\mathscr{C}$. For example, the term \textit{inverse limit of covering projections} will refer to maps of the form $p$ where each $p_j$ is a covering projection.
\end{definition}

It is straightforward to see that the class of maps with the continuous path-covering property are closed under inverse limits in the above sense. Moreover, the following lemma requires a direct argument involving the universal property of inverse limits; see the proof of \cite[Lemma 2.31]{Brazcat} for details.

\begin{lemma}\label{invlimitlemma}
Suppose $J$ is a directed set, $p_j:E_j\to X$ is a map with the continuous path-covering property for every $j\in J$, and $f_{j,j'}:E_j\to E_j'$ are maps satisfying $p_{j'}\circ f_{j,j'}=p_{j}$ whenever $j\geq j'$. Let $(e_j)_{j\in J}\in \varprojlim_{j}(E_j,f_{j,j'})$, $E$ be the path component of $(e_j)_{j\in J}$, and $p:E\to X$ be the restriction of $\varprojlim_{j}p_j:\varprojlim_{j}E_j\to X$. Then $p$ has the continuous path-covering property and corresponds to the subgroup $p_{\#}(\pi_1(E,(e_j)))=\bigcap_{j\in J}(p_j)_{\#}(\pi_1(E_j,e_j))$.
\end{lemma}

\begin{theorem}\label{invlimittheorem}
If $X$ is locally path connected, then a subgroup $H\leq\pionex$ is the intersection of open normal subgroups if and only if there exists an inverse limit $p:(E,e_0)\to (X,x_0)$ of regular covering projections that corresponds to $H$.
\end{theorem}
\begin{proof}
We recall \cite[Corollary 5.9]{FZ13corefree} which states that a subgroup $K\leq \pionex$ contains an open normal subgroup if and only if there exists a covering projection over $X$ that corresponds to $K$. As a consequence, if $q:(E,e)\to (X,x)$ is a covering projection, then $q_{\#}(\pi_1(E,e))$ is open in $\pi_1(X,x)$.

Suppose $H=\bigcap_{j\in J}N_j$ where $N_j$ is an open normal subgroup of $\pionex$. We may assume $\{N_j\mid j\in J\}$ is the set of all open normal subgroups in $\pionex$ containing $H$ so that $J$ becomes a directed set: $j\geq j'$ in $J$ if and only if $N_j\leq N_{j'}$. For each $j\in J$, find a regular covering projection $p_j:(E_j,e_j)\to (X,x_0)$ corresponding to $N_j$. Since the spaces $E_j$ are locally path connected, the usual lifting properties of covering projections give the existence of unique maps $f_{j,j'}:(E_j,e_j)\to (E_{j'},e_{j'})$ such that $p_{j'}\circ f_{j,j'}=p_j$ whenever $j\geq j'$. Together, these maps form an inverse system and taking the limit gives a map $\varprojlim_{j}p_j:\varprojlim_{j}E_j\to X$. Let $E$ be the path component of $(e_j)_{j\in J}$ in $\varprojlim_{j}E_j$. Consider the restriction $f_j:E\to E_j$ the projection map, and the restriction $p:
(E,(e_j))\to (X,x_0)$ of $\varprojlim_{j}p_j$. By Lemma \ref{invlimitlemma}, $p$ has the continuous path-covering property and $p_{\#}(\pi_1(E,(e_j)))=\bigcap_{j\in J}(p_j)_{\#}(\pi_1(E_j,e_j))=\bigcap_{j\in J}N_j=H$.

For the converse, suppose $p:E\to X$ is an inverse limit of regular covering projections $p_j:E_j\to X$ where $p:(E,e)\to (X,x)$ corresponds to $H$. Let $f_j:E\to E_j$ be the projections. Set $e_j=f_j(e)$ and $N_j=(p_j)_{\#}(\pi_1(E_j,e_j))$. Since $p_j$ is a regular covering projection, $N_j$ is an open normal subgroup of $\pi_1(X,x_0)$ for all $j\in J$. It follows from Lemma \ref{invlimitlemma} that $H=\bigcap_{j\in J}N_j$.
\end{proof}

\begin{definition}\cite{Brsemi}
A \textit{semicovering map} is a local homeomorphism $p:E\to X$ with the continuous path-covering property.
\end{definition}
We refrain from calling a semicovering a ``projection" since a semicovering need not be locally trivial. As observed in \cite{KMTSemicover}, one may define a semicovering to be a local homeomorphism with the path-covering property (the continuous path-covering property follows as a consequence). Every covering projection over a space $X$ is a semicovering. The converse rarely holds, e.g. even for the Hawaiian earring \cite{FZ13corefree}. We prove the following theorem, which generalizes the classical theorem that every covering projection is a Hurewicz fibration \cite[Theorem 2.2.3]{Spanier66}; the lack of local triviality requires us to formulate a line of argument different from Spanier's proof.

\begin{theorem}
Every semicovering map is a Hurewicz fibration with discrete fibers.
\end{theorem}
\begin{proof}
Let $p:E\to X$ be a semicovering map. Since $p$ is a local homeomorphism, $p$ has discrete fibers. It suffices to check that $p$ has the homotopy lifting property with respect to an arbitrary space $Z$. Let $f:Z\to E$ and $g:Z\times I\to X$ be maps such that $p(f(z))=g(z,0)$. For each $z\in Z$, let $\gamma_z:I\to X$ denote the path given by $g(z,t)$ and let $\wt{\gamma}_z:I\to E$ be the unique continuous lift such that $\wt{\gamma}_z(0)=f(z)$. This gives a function $\wt{g}:Z\times I\to  E$ defined by $\wt{g}(z,t)=\wt{\gamma}_z(t)$. Since $p\circ \wt{g}=g$, it suffices to show that $\wt{g}$ is continuous. We do this by showing that $\wt{g}$ is continuous on each member of an open cover of $Z\times I$ by sets of the form $V\times I$. 

Fix $z_0\in Z$. Since the path $\wt{\gamma}_{z_0}:I\to E$ is continuous, we may find a subdivision $0=t_0<t_1<t_2<\dots <t_m=1$ and open sets $U_1,U_2,\dots,U_m$ such that $\wt{\gamma}_{z_0}([t_{j-1},t_j])\subseteq U_j$ and such that $p$ maps $U_j$ homeomorphically onto the open set $p(U_j)$ of $X$. Find an open neighborhood $A_1$ of $z_0$ in $Z$ such that $f(A_1)\subseteq U_1$. Since $\gamma_{z_0}([t_0,t_1])\subseteq p(U_1)$, the compactness of $[t_0,t_1]$ and the continuity of $g$ allow us to find a neighborhood $B_1$ of $z_0$ in $Z$ such that $B_1\subseteq A_1$ and $g(B_1\times [t_0,t_1])\subseteq p(U_1)$. Since $\wt{g}(B_1\times \{t_0\})=f(B_1)\subseteq U_1$, we have $\wt{g}|_{V_1'\times [t_0,t_1]}=p|_{U_j}^{-1}\circ g|_{B_1\times [t_0,t_1]}$ and so we may conclude that $\wt{g}$ is continuous on $B_1\times [t_0,t_1]$. Since $g(z_0,t_1)=\gamma_{z_0}(t_1)\in p(U_1\cap U_2)$, we may find a neighborhood $A_2$ of $z_0$ in $Z$ such that $A_2\subseteq B_1$ and $g(A_2\times \{t_1\})\subseteq p(U_1\cap U_2)$. Since $p$ maps $U_1\cap U_2$ homeomorphically onto the open set $p(U_1\cap U_2)$ (but not necessarily onto $p(U_1)\cap p(U_2)$), our choice ensures that $\wt{g}$ is continuous on $A_2\times [t_0,t_1]$ and $\wt{g}(A_2\times \{t_1\})\subseteq U_1\cap U_2$. 

Applying the same procedure, we may find neighborhoods $z_0\in A_3\subseteq B_2\subseteq A_2$ such that $\wt{g}$ is continuous on $A_3\times [t_1,t_2]$ and $\wt{g}(A_3\times \{t_2\})\subseteq U_2\cap U_3$. Proceeding inductively, we obtain finitely many nested neighborhoods $z_0\in A_n\subseteq A_{n-1}\subseteq \cdots \subseteq A_2\subseteq A_1$ such that $\wt{g}$ is continuous on $A_{j}\times [t_{j-2},t_{j-1}]$ and $\wt{g}(A_j\times \{t_{j-1}\})\subseteq U_{j-1}\cap U_j$ (the second inclusion being required for the induction). We terminate the induction with $A_{n+1}$ by taking $U_{n+1}=U_n$. 

Let $V=A_{n+1}$. By restricting $\wt{g}$, we see that $\wt{g}$ is continuous on $V\times [t_{j-1},t_j]$ for all $j\in\{1,2,\dots, n\}$. Hence, by the pasting lemma, $\wt{g}$ is continuous on the tube $V\times I$. Since $\wt{g}$ is continuous on an open neighborhood of every point in $Z\times I$, $\wt{g}$ is continuous.
\end{proof}

Since the class of Hurewicz fibrations with totally path-disconnected fibers is closed under inverse limits (in the sense of Definition \ref{inverselimitsdef}), we have the following corollary.

\begin{corollary}\label{semicoveringcorollary}
If $p:E\to X$ is an inverse limit of semicovering maps, then $p$ is a Hurewicz fibration.
\end{corollary}

\begin{remark}\label{semicoveringremark}
If $p:(E,e_0)\to (X,x_0)$ is any semicovering, then $p_{\#}(\pi_1(E,e_0))$ is an open subgroup of $\pionex$ \cite[Proof of Theorem 5.5]{Brsemi}. For a locally path-connected space $X$ (and many non-locally path-connected spaces), the open subgroups of $\pionex$ are classified by the semicovering maps over $X$ \cite{Brsemi}. By mimicking the proof of Theorem \ref{invlimittheorem}, one may prove an analogous statement for inverse limits: $H$ is the intersection of open (not necessarily normal) subgroups if and only if there exists an inverse limit $p:(E,e_0)\to (X,x_0)$ of semicovering maps that corresponds to $H$.
\end{remark}
\begin{definition}\cite{BrazFGasTopGrp}
The \textit{$\tau$-topology} on $\pi_1(X,x_0)$ is the finest topology such that (1) $\pionex$ is a topological group and (2) the map $q:\Omega(X,x_0)\to \pionex$, $q(\alpha)=[\alpha]$ is continuous. We write $\pi_{1}^{\tau}(X,x_0)$ for the fundamental group equipped with the $\tau$-topology.
\end{definition}
Note that the quotient topology of $\pionex$ is generally finer than that of $\pi_{1}^{\tau}(X,x_0)$ and we have $\pi_1(X,x_0)=\pi_{1}^{\tau}(X,x_0)$ if and only if $\pionex$ is a topological group. It is shown in \cite[Cor. 3.9]{BrazFGasTopGrp} that the groups $\pionex$ and $\pi_{1}^{\tau}(X,x_0)$ share the same open subgroups. Additionally, every topological group $G$ is isomorphic to $\pi_{1}^{\tau}(X,x_0)$ for some space $X$ where $X$ is constructed in a manner similar to a 2-dimensional CW-complex, that is by attaching 2-cells to a generalized wedge of circles. In particular, such spaces $X$ satisfy a property called ``wep-connectedness," which is introduced in \cite{Brsemi}, and which lies between $X$ being locally path-connected and $ev_1:P(X,x_0)\to X$ being quotient \cite[Prop. 6.2]{Brsemi}. We will only need to use the following weaker fact: every topological group $G$ is isomorphic to $\pi_{1}^{\tau}(X,x_0)$ for a space $X$ where $ev_1:P(X,x_0)\to X$ is a quotient map.
\begin{example}
We give an extreme example to show that there are maps with the continuous path-covering property that cannot be inverse limits of covering projections. Find a space $X$ such that $ev_1:P(X,x_0)\to X$ is quotient and $\pi_{1}^{\tau}(X,x_0)$ is topologically isomorphic the additive group of rationals $\bbq$. Since the quotient topology of $\pionex$ is finer than the $\tau$-topology, $\pionex$ is totally path-disconnected. Moreover, if $H\leq \pionex$ is closed, then $\pionex/H$ is a countable $T_1$ space and therefore must be totally path-disconnected. Therefore, the closed subgroups of $\bbq$ are classified by maps $E\to X$ with the continuous path-covering property up to weak equivalence (and up to equivalence if we restrict to total spaces $E$ with $ev_1:P(E,e)\to E$ quotient). However, $\bbq$ has no proper open subgroups. Since $\pionex$ and $\pi_{1}^{\tau}(X,x_0)$ share the same open subgroups, $\pionex$ has no proper open subgroups. Therefore, $X$ admits many maps $E\to X$ with the continuous path-covering property but Remark \ref{semicoveringremark} ensures that the identity map $X\to X$ is the only one which is equivalent to a (semi)covering or inverse limit of (semi)coverings.
\end{example}
We consider inverse limits of (semi)coverings to be ``highly structured" among those maps with the continuous path-covering property. Recall the statement of Theorem \ref{mainresult2} from the introduction. To prove this result, we apply some famous structure theorems from topological group theory. 
\begin{proof}[Proof of Theorem \ref{mainresult2}]
We consider only the non-trivial case where $X$ is not simply connected. Since $p$ has the continuous path-covering property, $\pionex/H$ is a  path-disconnected, $T_1$ quasitopological group (recall Proposition \ref{tpdfibersprop} and Corollary \ref{tpdcorollary}). If $\pionex/H$ is compact or, more generally, locally compact, then the quasitopological group $\pionex/H$ is a topological group by a theorem of R. Ellis \cite{ellis}. Hence, in both statements (1) and (2) to be proved, $\pionex/H$ is a non-trivial, locally compact, totally path-disconnected, Hausdorff topological group. It is a result of Gleason \cite{Gleason} that every locally compact group that is not totally disconnected must contain an arc. Hence, $\pionex/H$ must also be totally disconnected.

(1) It is well-known that every totally disconnected compact group is a pro-finite group \cite[Theorem 1.34]{compactgroups}. Therefore, if $\pionex/H$ is compact, it is profinite and the identity element of $\pionex/H$ has a basis $\{N_j\mid j\in J\}$ of finite-index open normal subgroups (whose intersection is the trivial subgroup). If $q:\pionex\to \pionex/H$ is the natural quotient map, then it follows that $H$ is equal to the intersection $\bigcap_{j\in J}q^{-1}(N_j)$ of open subgroups $q^{-1}(N_j)$, which are finite-index and normal in $\pionex$. Since $X$ is locally path-connected, by \cite[Corollary 5.9]{FZ13corefree}, there exists finite sheeted regular covering maps $p_j:E_j\to X$ corresponding to the subgroups $q^{-1}(N_j)$ respectively.
Applying Theorem \ref{invlimittheorem}, we see that $p$ is weakly equivalent to the restriction of an inverse limit of finite-sheeted regular covering projections.

(2) If $E$ is simply connected, or equivalently if $H=1$, and $\pionex=\pionex/H$ is locally compact, then by van Dantzig's Theorem \cite{vandantzig}, there exists a compact open subgroup $K\leq \pionex$. Since $X$ is locally path-connected, the classification of semicoverings \cite{Brsemi} applies and there exists a semicovering $q:(E',e_0')\to (X,x_0)$ such that $K=q_{\#}(\pi_1(E',e_0'))$. Note that $E'$ is locally path connected since $q$ is a local homeomorphism. By Statement (2) of Theorem \ref{mainresult1}, we may assume that $ev_1:P(E,e_0)\to E$ is quotient without changing the weak equivalence class of $p$. Since $E$ is simply connected, Lemma \ref{liftinglemma} gives a unique map $r:(E,e_0)\to (E',e_0')$ such that $q\circ r=p$. Since $p$ and $q$ have the continuous path-covering property, it follows from Lemma \ref{compositionlemma} that $r$ also has the continuous path-covering property. By Theorem \ref{closedembeddingonpi1}, $q_{\#}$ maps $\pi_1(E',e_0')$ homeomorphically onto $K$. In particular, $\pi_1(E',e_0')$ is compact. By Part (1), $r$ is weakly equivalent to the restriction of an inverse limit of covering projections $r_j:(E_j,e_j)\to (E',e_{0}')$, $j\in J$ over $E'$. The composition of two semicoverings is a semicovering \cite[Cor. 3.5]{Brsemi} (but need not be a covering projection) and thus $q\circ r_j:E_j\to X$ is a semicovering for all $j\in J$. Since $\bigcap_{j\in J}(r_j)_{\#}(\pi_1(E_j,e_j))=1$ and $q_{\#}$ is injective, we have \[\bigcap_{j\in J}(q\circ r_j)_{\#}(\pi_1(E_j,e_j))=q_{\#}\left(\bigcap_{j\in J}(r_j)_{\#}(\pi_1(E_j,e_j))\right)=1.\]
Applying Lemma \ref{invlimitlemma} and Theorem \ref{mainresult1}, it follows that $p$ is weakly equivalent to the inverse limit of the semicovering maps $q\circ r_j$
\end{proof}
%
%

\section{Diagramatic Summary}\label{sectiondiagram}

We conclude with a diagram that summarizes the relationships between lifting properties studied in this paper and the topology of $\pionex$. In the diagram below $X$ is assumed to be locally path-connected, $p:(E,e_0)\to (X,x_0)$ is a map and $H=p_{\#}(\pi_1(E,e_0))$. The left side of the diagram involves properties of $p$ and the right side of the diagram involves properties of $H$. We give the following key for reading the diagram:
\begin{itemize}
\item $A\leq_{cl}B$ - $A$ is a closed subgroup of $B$
\item $A\leq_{op}B$ - $A$ is an open subgroup of $B$
\item $A\trianglelefteq_{op}B$ - $A$ is an open normal subgroup of $B$
\item t.p.d. - totally path-disconnected
\item $Core_G(H)=\bigcap_{g\in G}gHg^{-1}$ - the core of $H$ in $G$, i.e. the largest normal subgroup of $G$ that is a subgroup of $H$.
\end{itemize}
A horizontal biconditional arrows means that there exists a map weakly equivalent to $p$ that satisfies the property on the left if and only if $H$ satisfies the property on the right. For example, $p$ is weakly equivalent to a covering projection if and only if the core of $H$ in $G$ is open, which is equivalent to shorter statement $Core_G(H)\trianglelefteq_{op}H$. 

All horizontal or downward arrows hold without extra hypotheses. The two partial converse arrows that point upward (from Theorem \ref{mainresult2}) are labeled with the extra required hypotheses (indicated with a ``$+$") and appear on opposite sides of the diagram to minimize clutter in the image.

\[\xymatrix{
\scriptsize{\txt{covering \\ projection}} \ar@{<=>}[rrr]^-{\scriptsize{\txt{\cite[Cor. 5.9]{FZ13corefree}}}} \ar@{=>}[d]_-{\tiny{\txt{\cite[Prp. 3.7]{Brsemi}}}} \ar@{=>}[ddr] && &  \scriptsize{\txt{$Core_G(H)\trianglelefteq_{op} H $}} \ar@{=>}[d] \ar@{=>}[ddr] \\
\scriptsize{\txt{semicovering\\ map}}  \ar@{<=>}[rrr]^-{\scriptsize{\txt{\cite[Cor. 7.20]{Brsemi}}}}|!{[u];[dr]}\hole \ar@{=>}[d]  &   && 
\scriptsize{\txt{$H\leq_{op} G$}} \ar@{=>}[d]  &     \\
\scriptsize{\txt{inv. limit of\\
semicoverings}} \ar@/^1.8pc/@{<=>}[rrr]^-{\tiny{\txt{Rmk. \ref{semicoveringremark}}}}|!{[uu];[r]}\hole   \ar@{=>}[d]_-{\tiny{\txt{Cor. \ref{semicoveringcorollary}}}} & \scriptsize{\txt{inv. limit of\\
covering \\ projections}}  \ar@{=>}[l] \ar@{=>}[dl] \ar@/_1.8pc/@{<=>}[rrr]_(.4){\tiny{\txt{Thm. \ref{invlimittheorem}}}}|!{[urr];[ddrr]}\hole  &&  \tiny{\txt{$\ds H= \bigcap_{H\leq K\leq_{op}G}K$}} \ar@<-.5ex>@{=>}[dd]& \tiny{\txt{$ \ds H=\bigcap_{H\leq N\trianglelefteq_{op}G}N$}}  \ar@{=>}[l]
\\
\scriptsize{\txt{Hur. fibration\\ + t.p.d. fibers}} \ar@{=>}[d]_-{\tiny{\txt{Lem. \ref{hurfibrationlemma}}}}\\
\scriptsize{\txt{continuous\\ path-covering\\property}}  \ar@{=>}[uur]_-{\tiny{\txt{+ $G/H$ a compact grp. \\
Thm. \ref{mainresult2} (1)}}} \ar@{<=>}[rrr]_-{\tiny{\txt{Thm. \ref{mainresult1}}}} \ar@{=>}[d]_-{\tiny{\txt{Thm. \ref{fibrationtheorem}}}} &&& \scriptsize{\txt{$G/H$ t.p.d.}}  \ar@<-.5ex>@{=>}[uu]_(.4){\tiny{\txt{+ $H=1$ \\+ $G$ loc. compact\\
Thm. \ref{mainresult2} (2)}}}   \ar@{=>}[d]  \\
\scriptsize{\txt{Serre fibration\\ + t.p.d. fibers}}  &&& \scriptsize{\txt{$H\leq_{cl}G$}} \ar@{=>}[lll]^-{\scriptsize{\txt{\cite[Thm. 11]{BFqtpfg}}}}
}\]

\begin{thebibliography}{10}
\expandafter\ifx\csname url\endcsname\relax
  \def\url#1{\texttt{#1}}\fi
\expandafter\ifx\csname urlprefix\endcsname\relax\def\urlprefix{URL }\fi




\bibitem{ATHP99}
S.~Ardanza-Trevijana, L.-J.~Hernández-Paricio, Fundamental progroupoid and bundles with a structural category, Topology Appl. {\bf 92} (1999) 85--99.

\bibitem{BP01topgrp} V.~Berestovskii, C.~Plaut, {\it Covering group theory for topological groups}, Topology Appl. {\bf 114} (2001) 141--186.

\bibitem{BP07uniform} V.~Berestovskii, C.~Plaut, {\it Uniform universal covers of uniform spaces}, Topology Appl. {\bf 154} (2007) 1748--1777.

\bibitem{BrazFGasTopGrp} J.~Brazas, \textit{The fundamental group as a topological group}, Topology Appl. {\bf 160} (2013) 170--188.

\bibitem{Brsemi} J.~Brazas, \textit{Semicoverings: A generalization of covering space theory}, Homotopy, Homology Appl. {\bf 14} (2012) 33--63.

\bibitem{BrazOpenSubgroupsofFTG} J. Brazas, \textit{Open subgroups of free topological groups}, Fund. Math. {\bf 226} (2014) 17--40.


\bibitem{Brazcat} J.~Brazas, \emph{Generalized covering space theories}, Theory and Appl. of Categories {\bf 30} (2015) 1132--1162.

\bibitem{BFqtpfg} J. Brazas, P. Fabel, \textit{On fundamental groups with the quotient topology}, Journal of Homotopy and Related Structures {\bf 10} (2015) 71--91.

\bibitem{BFTestMap} J.~Brazas, H.~Fischer, \emph{Test map characterizations of local properties of fundamental groups}, Journal of Topology and Analysis. {\bf 12} (2020) 37--85.

\bibitem{BDLM10uniform} N.~Brodskiy, J.~Dydak, B.~Labuz, A.~Mitra, {\it Group actions and covering maps in the uniform category}, Topology Appl. {\bf 157} (2010) 2593--2603.

\bibitem{BDLM08} N.~Brodskiy, J.~Dydak, B.~Labuz, A.~Mitra, \emph{Covering
maps for locally path connected spaces}, Fund. Math. {\bf 218} (2012) 13--46.

\bibitem{CM}
J.~Calcut, J.~McCarthy, Discreteness and homogeneity of the topological
  fundamental group, Topology Proc. {\bf 34} (2009) 339--349.

\bibitem{CC06} J.~Cannon, G.~Conner, On the fundamental group of one-dimensional spaces, Topology Appl. {\bf 153} (2006) 2648--2672.

\bibitem{CHPGeomofCLS} G.~Conner, W.~Herfort, P.~Pave\v{s}i\'{c}, \textit{Geometry of compact lifting spaces}. Preprint. arXiv:1901.02108.


\bibitem{Dydak} J. Dydak, \textit{Coverings and fundamental groups: a new approach}, Preprint. arXiv:1108.3253v1. 2011.



\bibitem{Eda98} K.~Eda, K.~Kawamura, \textit{The fundamental groups of one-dimensional spaces}, Topology Appl. {\bf 87} (1998) 163--172.

\bibitem{ellis} R.~Ellis, \textit{Locally compact transformation groups}. Duke Math. J. {\bf 24} (1957) 119--125.

\bibitem{engelking} R. Engelking, \textit{General topology}, Heldermann Verlag Berlin, 1989.

\bibitem{Fab10} P.~Fabel, \textit{Multiplication is discontinuous in the Hawaiian earring group}, Bull. Polish Acad. Sci. Math. {\bf 59} (2011) 77--83

\bibitem{Fab11} P.~Fabel, \textit{Compactly generated quasitopological homotopy groups with discontinuous multiplication}, Topology Proc. {\bf 40} (2012) 303--309.

\bibitem{FZ05} H.~Fischer, A.~Zastrow, \textit{The fundamental groups of subsets of closed surfaces inject into their first shape groups}, Algebr. Geom. Topol. {\bf 5} (2005) 1655–-1676.

\bibitem{FG05} H.~Fischer, C.~Guilbault, \textit{On the fundamental groups of trees of manifolds}, Pacific J. Math. {\bf 221} (2005) 49--79.


\bibitem{FZ07} H.~Fischer, A.~Zastrow, \emph{Generalized universal covering spaces and the shape group}, Fund. Math. {\bf 197} (2007) 167--196.

\bibitem{FZ13corefree} H.~Fischer, A.~Zastrow, \emph{ A core-free semicovering of the Hawaiian Earring}, Topology Appl. {\bf 160} (2013) 1957--1967.

\bibitem{Fox74} R.~Fox, \textit{Shape theory and covering spaces}, Vol. 375 of Lecture Notes in Mathematics, Springer Verlag, 1974.

\bibitem{GHMMTopHomGrps} H. Ghane, Z.~Hamed, B. Mashayekhy, H.~Mirebrahimi, \textit{Topological Homotopy Groups}, Bull. Belg. Math. Soc. {\bf 15} (2008) 455--464.

\bibitem{Gleason} A.M.~Gleason. \textit{Arcs in locally compact groups}, Proc. Nat. Acad. Sci. U.S.A. {\bf 36} (1950), 663--667.

\bibitem{HP98} L.-J.~Hernández-Paricio, Fundamental pro-groupoids and covering projections, Fund. Math. {\bf 156} (1998) 1--33.

\bibitem{HPM} L. J. Hern\' andez Paricio, V. Matijevi\' c, {\it Fundamental groups and finite sheeted coverings}, J. Pure Appl. Algebra {\bf 214} (2010) 281--296.


\bibitem{HW60} P.~Hilton, S.~Wylie, Homology Theory: An introduction to algebraic topology, Cambridge University Press, 1960.

\bibitem{compactgroups} K.H.~Hofmann, S.A.~Morris, \textit{The structure of compact groups}, De Gruyter Studies in Mathematics vol. 25. De Gruyter, Berlin, 2013.

\bibitem{Hovey} M. Hovey, \emph{Model Categeories}, Mathematical Surveys and Monographs, vol. 63, American Mathematical Society, Providence, RI, 1998. 


\bibitem{KMTSemicover} M.~Kowkabi, B.~Mashayekhy, H.~Torabi, \emph{When is a Local Homeomorphism a Semicovering Map?}, Acta Mathematica Vietnamica {\bf 42} (2017) 653–-663.



\bibitem{Spanier66} E.~Spanier, \textit{Algebraic Topology}, McGraw-Hill, 1966.

\bibitem{vandantzig} D. van Dantzig, \textit{Zur topologischen Algebra. III. Brouwersche und Cantorsche Gruppen}, Compositio Mathematica {\bf 3} (1936) 408–-426.

\bibitem{whiteheadquotient} J.H.C.~Whitehead, \textit{A note on a theorem of Borsuk}, Bull. Amer. Math. Soc. {\bf 54} (1958) 1125--1132.

\bibitem{Wilkins} Jay Wilkins, {\it The revised and uniform fundamental groups and universal covers of geodesic spaces}, Topology Appl. {\bf 160} (2013), no. 6, 812--835.

\end{thebibliography}
\end{document}